\documentclass{amsart}

\usepackage[utf8]{inputenc}
\usepackage[T1]{fontenc}
\usepackage{amscd,amssymb,amsfonts,accents}
\usepackage{hyperref}
\usepackage[lmargin=1in,rmargin=1in,tmargin=1in,bmargin=1in]{geometry}

\usepackage{color}
\usepackage[mathscr]{euscript}
\usepackage[all]{xy}
\usepackage{fancyhdr}

\allowdisplaybreaks[3]

\numberwithin{equation}{section}

\newtheorem{prop}{Proposition}[section]
\newtheorem{lemma}[prop]{Lemma}
\newtheorem{thm}[prop]{Theorem}
\newtheorem{cor}[prop]{Corollary}

\theoremstyle{definition}
\newtheorem{defn}[prop]{Definition}

\newtheorem{rmk}[prop]{Remark}
\newtheorem{quest}[prop]{Question}
\newtheorem{ex}[prop]{Example}

\DeclareMathOperator{\rep}{Rep}
    
\DeclareMathOperator{\spec}{Spec} 

\DeclareMathOperator{\im}{Im}
\DeclareMathOperator{\Gal}{Gal}

\DeclareMathOperator{\proj}{Proj}
\DeclareMathOperator{\Stab}{Stab} 

\DeclareMathOperator{\End}{End}
\DeclareMathOperator{\Aut}{Aut}

\DeclareMathOperator{\brep}{\textbf{Rep}}
\DeclareMathOperator{\baut}{\textbf{Aut}}

\newcommand{\et}{\mathrm{\acute{e}t}}
\DeclareMathOperator{\Pic}{Pic}

\newcommand{\Gbar}{\overline{\mathbf{G}}_{Q,d}}

\newcommand{\Mods}{\mathcal{M}_{Q,d}^{\theta-s}}

\newcommand{\ra}{\rightarrow}

\def\cA{\mathcal A}
\def\cE{\mathcal E}\def\cF{\mathcal F}\def\cG{\mathcal G}
\def\cL{\mathcal L}
\def\cM{\mathcal M}\def\cO{\mathcal O}\def\cP{\mathcal P}
\def\cQ{\mathcal Q}\def\cR{\mathcal R}\def\cT{\mathcal T}
\def\cW{\mathcal W}

\def\AA{\mathbb A}\def\CC{\mathbb C}
\def\GG{\mathbb G}\def\HH{\mathbb H}

\def\NN{\mathbb N}
\def\RR{\mathbb R}

\def\ZZ{\mathbb Z}

\def\fM{\mathfrak M}

\def\fX{\mathfrak X}
\def\fZ{\mathfrak Z}

\newcommand{\op}{\mathrm{op}}
\newcommand{\Sets}{\mathrm{Sets}}

\newcommand{\Mat}{\mathrm{Mat}}
\newcommand{\Hom}{\mathrm{Hom}}
\newcommand{\Br}{\mathrm{Br}}
\newcommand{\ind}{\mathrm{ind}}

\def\GL{\mathbf{GL}}

\def\PGL{\mathbf{PGL}}

\renewcommand{\cF}{\mathcal{F}}

\newcommand{\R}{\mathbb{R}}
\newcommand{\Q}{\mathbb{Q}}
\newcommand{\C}{\mathbb{C}}
\newcommand{\N}{\mathbb{N}}
\newcommand{\A}{\mathbb{A}}
\newcommand{\Z}{\mathbb{Z}}

\newcommand{\lra}{\longrightarrow}
\newcommand{\lmt}{\longmapsto}
\newcommand{\Mod}{\mathcal{M}_{Q,d}^{\theta-ss}}
\newcommand{\Modgs}{\mathcal{M}_{Q,d}^{\theta-gs}}
\newcommand{\Rep}{\rep_{Q,d}}
\newcommand{\G}{\mathbf{G}_{Q,d}}
\newcommand{\pphi}{\varphi}
\newcommand{\Fss}{F_{Q,d}^{\theta-ss}}
\newcommand{\Fgs}{F_{Q,d}^{\theta-gs}}
\newcommand{\Gm}{\mathbb{G}_{m}}

\newcommand{\Repss}{\rep_{Q,d}^{\chi_\theta-ss}}
\newcommand{\Repgs}{\rep_{Q,d}^{\chi_\theta-s}}
\newcommand{\kb}{\overline{k}}
\newcommand{\fGalk}{f_{\Gal_k}}
\newcommand{\chith}{\chi_\theta}
\newcommand{\ka}{\kappa}

\newcommand{\Reps}{\rep_{Q,d}^{\theta-gs}}

\newcommand{\crep}{{\cR}ep}
\newcommand{\cmod}{{\cM}od}

\newcommand{\ov}[1]{\overline{#1}}

\title{Rational points of quiver moduli spaces}
\author{Victoria Hoskins}
\address{Freie Universit\"at Berlin, Arnimallee 3, Raum 011, 14195 Berlin, Germany.}
\email{hoskins@math.fu-berlin.de}
\author{Florent Schaffhauser}
\address{Universidad de Los Andes, Carrera 1 \#18A-12, 111 711 Bogot\'a, Colombia.}
\email{florent@uniandes.edu.co}

\thanks{The authors thank the Institute of Mathematical Sciences of the National University 
of Singapore, where part of this work was carried out, for their hospitality in 2016, and acknowledge the support from U.S. National Science Foundation grants DMS 1107452, 1107263, 1107367 "RNMS: Geometric structures And Representation varieties" (the GEAR Network). The first author is supported by the Excellence Initiative of the DFG at the Freie Universit\"{a}t Berlin.}

\keywords{Algebraic moduli problems (14D20), Geometric Invariant Theory (14L24), Representations of quivers (16G20)}

\begin{document}

\begin{abstract}
For a perfect base field $k$, we investigate arithmetic aspects of moduli spaces of quiver representations over $k$: we study actions of the absolute Galois group of $k$ on the $\overline{k}$-valued points of moduli spaces of quiver representations over $k$ and we provide a modular interpretation of the fixed-point set using quiver representations over division algebras, which we reinterpret using moduli spaces of twisted quiver representations (we show that those spaces provide different $k$-forms of the initial moduli space of quiver representations). Finally, we obtain that stable $\overline{k}$-representations of a quiver are definable over a certain central division algebra over their field of moduli.
\end{abstract}

\maketitle

\tableofcontents

\section{Introduction}

For a quiver $Q$ and a field $k$, we consider moduli spaces of semistable $k$-represen\-tations of $Q$ of fixed dimension $d \in \NN^V$, which were first constructed for an algebraically closed field $k$ using geometric invariant theory (GIT) by King \cite{king}. For an arbitrary field $k$, one can use Seshadri's extension of Mumford's GIT to construct these moduli spaces. More precisely, these moduli spaces are constructed as a GIT quotient of a reductive group $\G$ acting on an affine space $\Rep$ with respect to a character $\chi_\theta$ determined by a stability parameter $\theta \in \ZZ^V$. The stability parameter also determines a slope-type notion of $\theta$-(semi)stability for $k$-representations of $Q$, which involves testing an inequality for all proper non-zero subrepresentations. When working over a non-algebraically closed field, the notion of $\theta$-stability is no longer preserved by base field extension, so one must instead consider $\theta$-geometrically stable representations (that is, representations which are $\theta$-stable after any base field extension), which correspond to the GIT stable points in $\Rep$ with respect to $\chi_\theta$. 

We let $\Mod$ (resp.\ $\Modgs$) denote the moduli space of $\theta$-semistable (resp.\ $\theta$-geometrically stable) $k$-representations of $Q$ of dimension $d$; these are both quasi-projective varieties over $k$ and are moduli spaces in the sense that they corepresent the corresponding moduli functors (\textit{cf.}\ $\S$\ref{quiver_moduli_sect}). For a non-algebraically closed field $k$, the rational points of $\Modgs$ are not necessarily in bijection with the set of isomorphism classes of $\theta$-geometrically stable $d$-dimensional $k$-representations of $Q$. In this paper, we give a description of the rational points of this moduli space for perfect fields $k$. More precisely, for a perfect field $k$, we study the action of the absolute Galois group $\Gal_k = \Gal(\kb/k)$ on $\Mod(\kb)$, whose fixed locus is the set $\Mod(k)$ of $k$-rational points. 
We restrict the action of $\Gal_k$ to $\Modgs \subset \Mod$, so we can use the fact that the stabiliser of every GIT stable point in $\Rep$ is a diagonal copy of $\GG_m$, denoted $\Delta$, in $\G$ (\textit{cf.}\ Corollary \ref{stabiliser_of_GIT_stable_points}) to decompose the fixed locus of $\Gal_k$ acting on $\Modgs(\kb)$ in terms of the group cohomology of $\Gal_k$ with values in $\Delta$ or the (non-Abelian) group $\G$. 

Before we outline the main steps involved in our decomposition, we note that, via a similar procedure, we can describe fixed loci of finite groups of quiver automorphisms acting on quiver moduli spaces in \cite{HS_quiver_autos}; however, the decomposition for the $\Gal_k$-fixed locus in this paper is simpler than the decomposition given in \cite{HS_quiver_autos}, as we can use Hilbert's 90th Theorem to simplify many steps. When $k=\R$ and moduli spaces of quiver representations are replaced by moduli spaces of vector bundles over a real algebraic curve, a description of the real points of these moduli spaces has been obtained in \cite{Sch_JSG} using similar techniques: while the arithmetic aspects are much less involved in that setting, they still arise, in the guise of quaternionic vector bundles.

\subsection{Decomposition of the rational locus}

For the first step, we note that the $\Gal_k$-action on $\Mod(\kb)$ can be induced by compatible $\Gal_k$-actions on $\Rep$ and $\G$; then, we construct maps of sets
\[ \xymatrix@1{ \Rep^{\theta-gs}(\kb)^{\Gal_k}/\G(\kb)^{\Gal_k} \ar@{=}[d] \ar[r]^{ \quad \quad \quad \, f_{\Gal_k}} &  \Modgs(\kb)^{\Gal_k}  \ar@{=}[d]\\
\Rep^{\theta-gs}(k)/ \G(k)  \ar[r]^{ \quad \quad \quad f_{\Gal_k}}  & \Modgs(k). } \]
By Proposition \ref{fibres_of_fGalk},  $f_{\Gal_k}$ is injective, as a consequence of Hilbert's 90th Theorem.
 
The second step is to understand rational points not arising from rational representations by constructing
\[ \cT : \Modgs(k)=\Modgs(\kb)^{\Gal_k} \lra H^2(\Gal_k,\Delta(\kb)) \cong \Br(k) \]
which we call the type map (\textit{cf.}\ Proposition \ref{type_map}). We show that the image of $f_{\Gal_k}$ is $\cT^{-1}([1])$ (\textit{cf.}\ Theorem \ref{image_of_k_rep_in_k_pts_of_the_moduli_scheme}); however, in general $f_{\Gal_k}$ is not surjective.

In the third step, to understand the other fibres of the type map, we introduce the notion of a modifying family $u$ (\textit{cf.}\ Definitions \ref{modifying_fmly_def}) which determines  modified $\Gal_k$-actions on $\Rep(\kb)$ and $\G(\kb)$ such that the induced $\Gal_k$-action on $\Mod(\kb)$ is the original $\Gal_k$-action. Then we define a map of sets
\[ f_{\Gal_k,u} : \: _u\Rep^{\theta-gs}(\kb)^{\Gal_k}/ _u\G(\kb)^{\Gal_k}\lra  \Modgs(\kb)^{\Gal_k} = \Modgs(k),\]
where $ _u\Rep^{\theta-gs}(\kb)^{\Gal_k}$ and $ _u\G(\kb)^{\Gal_k}$ denote the fixed loci for the modified $\Gal_k$-action given by $u$. Moreover, we show that the image of $f_{\Gal_k,u}$ is equal to the preimage under the type map of the cohomology class of a $\Delta(\kb)$-valued 2-cocycle $c_u$ associated to $u$ (\textit{cf.}\ Theorem \ref{fibres_of_the_type_map}).

The fourth step is to describe the modifying families using the group cohomology of $\Gal_k$ in order to obtain a decomposition of the fixed locus for the $\Gal_k$-action on $\Modgs(\kb)$. We obtain the following decomposition indexed by the Brauer group $\Br(k)$ of $k$.

\begin{thm}\label{decomp_thm_gal_intro}
For a perfect field $k$, let $\cT:\Modgs(k) \lra H^2(\Gal_k;\kb^\times)\cong\Br(k)$ be the type map introduced in Proposition \ref{type_map}. Then there is a decomposition $$\Modgs(k) \simeq \bigsqcup_{[c_u]\in \mathrm{Im}\,\cT}  {_u}\Repgs(\kb)^{\Gal_k} /  _u\G(\kb)^{\Gal_k}$$ where 
$ _u\Repgs(\kb)^{\Gal_k} /  _u\G(\kb)^{\Gal_k}$ is the set of isomorphism classes of $\theta$-geome\-trically stable $d$-dimensio\-nal representations of $Q$ that are $k$-rational with respect to the twisted $\Gal_k$-action $\Phi^u$ on $\Rep(\kb)$ defined in Proposition \ref{modified_actions_Gal_case}.
\end{thm}

\subsection{Modular interpretation arising from the decomposition}

We give a modular interpretation of the decomposition above, by recalling that $\Br(k)$ can be identified with the set of central division algebras over $k$. We first prove that for a division algebra $D \in \Br(k)$ to lie in the image of the type map, it is necessary that the index $\ind\,(D):= \sqrt{\dim_k(D)}$ divides the dimension vector $d$ (\textit{cf.}\ Proposition \ref{nec_con_div_alg}). As a corollary, we deduce that if $d$ is not divisible by any of the indices of non-trivial central division algebras over $k$, then $\Modgs(k)$ is the set of isomorphism classes of $d$-dimensional $k$-representations of $Q$. We can interpret the above decomposition by using representations of $Q$ over division algebras over $k$.

\begin{thm}\label{thm_Galois_div_alg_intro}
Let $k$ be a perfect field. For a division algebra $D \in \mathrm{Im} \: \cT \subset \Br(k)$, we have $d = \ind\,(D) d'_D$ for some dimension vector $d'_D \in \NN^V$ and there is a modifying family $u_D$ and smooth affine $k$-varieties $\rep_{Q,d'_D,D}$ (resp.\ $\mathbf{G}_{Q,d'_D,D}$) constructed by Galois descent such that
\[ \rep_{Q,d'_D,D}(k) = \bigoplus_{a \in A} \Hom_{\cmod(D)}(D^{ d'_{D,t(a)}}, D^{ d'_{D,h(a)}}) = {_{u_D}}\rep_{Q,d}(\kb)^{\Gal_k}\]
and 
\[ \mathbf{G}_{Q,d'_D,D}(k) = \prod_{v \in V} \Aut_{\cmod(D)}(D^{ d'_{D,v}})={_{u_D}}\mathbf{G}_{Q,d}(\kb)^{\Gal_k}.  \]
Furthermore, we have a decomposition
\[ \Modgs(k) \cong  \bigsqcup_{D\in \mathrm{Im}\,\cT}  \rep_{Q,d'_D,D}^{\theta-gs}(k) / \mathbf{G}_{Q,d'_D,D}(k), \]
where the subset indexed by $D$ is the set of isomorphism classes of $d'_D$-dimensional $\theta$-geometrically stable $D$-representations of $Q$.
\end{thm}

For example, if $k = \RR \hookrightarrow \kb = \CC$, then as $\Br(\RR) = \{ \RR, \HH \}$, there are two types of rational points in $\Modgs(\RR)$, namely $\RR$-representations and $\HH$-representations of $Q$ and the latter can only exists if $d$ is divisible by $2 = \ind\,(\HH)$ (\textit{cf.}\ Example \ref{quaternionic_rep}).

\subsection{Gerbes and twisted quiver representations}

We can also interpret $\Br(k)$ as the set of isomorphism classes of $\GG_m$-gerbes over $\spec k$, and show that the type map $\cT$ can be defined for any field $k$ using the fact that the moduli stack of $\theta$-geometrically stable $d$-dimensional $k$-representations of $Q$ is a $\GG_m$-gerbe over $\Modgs$ (\textit{cf.}\ Corollary \ref{two_type_maps_agree}).

For any field $k$, we introduce a notion of twisted $k$-representations of a quiver $Q$ in Definition \ref{def_twisted_rep}, analogous to the notion of twisted sheaves due to C\u{a}ld\u{a}raru, de Jong and Lieblich  \cite{cald,dJ,Lieblich}, and we describe the moduli of twisted quiver representations. In particular, we show that twisted representations of $Q$ are representations of $Q$ over division algebras, by using C\u{a}ld\u{a}raru's description of twisted sheaves as modules over Azumaya algebras; therefore, the decomposition in Theorem \ref{thm_Galois_div_alg_intro} can also be expressed in terms of twisted quiver representations (\textit{cf.}\ Theorem \ref{decomp_twisted_reps}). Consequently, we construct moduli spaces of twisted $\theta$-geometrically stable $k$-representations of $Q$ and show, for Brauer classes in the image of the type map $\cT$, that these moduli spaces give different forms of the moduli space $\Modgs$. 

\begin{thm}\label{forms_of_moduli_sp}
For a field $k$ with separable closure $k^s$, let $\alpha : \fX \lra \spec k$ be a $\GG_m$-gerbe over $k$ and let $D$ be the corresponding central division algebra over $k$. Then the stack of  $\alpha$-twisted $\theta$-geometrically stable $d'$-dimensional $k$-representations
\[^\alpha \fM_{Q,d',k}^{\theta-gs}\cong[\rep_{Q,d',D}^{\theta-gs}/\mathbf{G}_{Q,d',D}]\]
is a $\GG_m$-gerbe over its coarse moduli space $\cM^{\theta-gs}_{Q,d',D}:=\rep_{Q,d',D}^{\theta-gs}/ \mathbf{G}_{Q,d',D}$ (in the sense of stacks). The moduli space $\cM^{\theta-gs}_{Q,d',D}$ is a coarse moduli space for: 
\begin{enumerate}
\item the moduli functor of $\theta$-geometrically stable $d'$-dimensional $D$-representa\-tions of $Q$, and 
\item the moduli functor of $\alpha$-twisted  $\theta$-geometrically stable $d'$-dimensional $k$-representations of $Q$.
\end{enumerate}
If, moreover, $D$ lies in the image of the type map $\cT$, then $d= \ind\,(D)d'$ for some dimension vector $d'$ and $\cM^{\theta-gs}_{Q,d',D}$ is a $k$-form of the moduli space $\cM^{\theta-gs}_{Q,d,k^s}$.
\end{thm}

As an application of these ideas, we define a Brauer class which is the obstruction to the existence of a universal family on $\Modgs$ and show that this moduli space admits a twisted universal family of quiver representations  (\textit{cf.}\ Proposition \ref{twisted_univ_family}).

\subsection{Applications to fields of moduli and fields of definition}

The problem that we study in the present paper is a special case of the rather general phenomenon that fields of moduli are usually smaller than fields of definition, which, to our knowledge, has not been worked out explicitly in the quiver setting. For a $\kb$-representation $W$ of $Q$, the field of moduli $k_W$ (with respect to the extension $\kb/k$) is the intersection of all intermediate fields $k \subset L \subset \kb$ over which $W$ is definable (that is, there exists an $L$-representation $W'$ such that $\kb \otimes_L W' \simeq W$). If there is a minimal field of definition, this is necessarily the field of moduli but, in general, a representation may not be definable over its field of moduli. 

This phenomenon was first studied in the context of algebraic curves and Abelian varieties (\cite{Weil_field_of_def,Matsusaka,Shimura_autom_functions}). By Weil's descent theorem, an algebraic curve with trivial automorphism group can be defined over its field of moduli. The situation is more complicated for Abelian varieties but similar rationality results hold in certain cases (\cite{Shimura_Taniyama,Shimura,Koizumi}). Quivers representations always have non-trivial automorphisms, thus it is natural to look for cohomological obstructions to a $\kb$-representation being definable over its field of moduli: this is for instance how the analogous problem for algebraic covers is studied in \cite{DD} (see also \cite{CH,DH,DDH}), whose results are then applied to curves with non-trivial automorphisms in \cite{DE}, by reducing the problem for a curve $X$ to the study of the algebraic cover $X\lra X/\Aut(X)$. The study of the cohomological obstructions that appear in the quiver setting (which we treat in detail in Sections \ref{Galois_actions_sect} and \ref{sec_gerbes_and_twisted_reps} of this paper) leads to the following result.

\begin{thm}\label{field_of_moduli_and_domain_of_def}
Let $k$ be a perfect field. Let $W$ be a $\theta$-stable $d$-dimensional $\kb$-representation of $Q$ and let $O_W\in \Modgs(\kb)$ be the associated point in the moduli space of $\theta$-geometrically stable $k$-representations. Then the field of moduli $k_W$ is isomorphic to the residue field $\kappa(O_W)$ and there is a central division algebra $D_W$ over $k_W$ and a $\theta$-geometrically stable $D_W$-representation $W'$ of $Q$, unique up to isomorphism, such that $\kb \otimes_{k_W} W' \simeq W$ as $\kb$-representations.
\end{thm}

\noindent Here we view $\kb \otimes_{k_W} W' \simeq W$ as a $\kb$-representation via Remark \ref{rmk_on_Morita_equiv}. The fact that the field of moduli of an object is isomorphic to the residue field of the corresponding point in the moduli space is a phenomenon that often occurs: it is for instance true for algebraic curves over a perfect field (\cite{Baily,Sekiguchi}), for Abelian varieties (in characteristic zero, the result for curves is deduced from the result for Abelian varieties via the Torelli theorem \cite{Baily}) and for algebraic covers (see \cite{RW}). The second part of Theorem \ref{field_of_moduli_and_domain_of_def} (the actual \lq domain of definition' of an object, knowing its field of moduli) is more delicate and depends on the explicit nature of the problem and the structure of the automorphism groups: an algebraic curve, for instance, is definable over a finite extension of its field of moduli (\cite{Koizumi,Huggins_thesis}), while for Abelian varieties, there holds an analogue of the second part of Theorem \ref{field_of_moduli_and_domain_of_def}, involving central division algebras over the field of moduli (see \cite{Shimura_div_alg_and_ab_var,Baily_report}). The cohomological obstruction to $W$ being defined over $k_W$ is described by Corollary \ref{cor_coh_obs_define_field_of_moduli}.

\subsection{Connections to species in representation theory and related open questions}

There is a close link between our work and the study of species in representation theory. The theory of species was introduced by Gabriel \cite{Gabriel} and studied by Dlab and Ringel \cite{RingelDlab} in order to classify finite-dimensional algebras over non-algebraically closed fields; in fact, species can be used to classify all hereditary algebras over perfect fields up to Morita equivalence \cite[Corollary 3.13]{Lemay}.

For a field $k$, a $k$-species $S$ of a quiver $Q$ is the data of a $k$-division algebra $S_v$ for each vertex $v \in V$ and a $(S_{h(a)}, S_{t(a)})$-bimodule $S_a$ for each arrow $s \in A$ (all of whose dimensions are controlled by a fixed valuation on $Q$). For a perfect field $k$ and quiver $Q$ without oriented cycles, the associated tensor algebra of a $k$-species $S$ is a $k$-form of the path algebra of a quiver $Q'$ up to Morita equivalence (\textit{cf.}\ \cite[Corollary 5.18]{Lemay}). If $k$ is also a finite field, then Hubery \cite{Hubery} shows that Morita equivalence can be replaced by isomorphism and there is a (covariant) automorphism $\sigma$ of $Q'$ that folds into $Q$ (thus $Q = Q'/\!\langle\sigma\rangle$ is the quotient quiver in the language of \cite{HS_quiver_autos}). Moreover, he extends Kac's description \cite{Kac} of the dimension vectors of indecomposible representations of a quiver $Q$ over a finite field as the positive roots of an associated Kac-Moody algebra to species over finite fields \cite{Hubery}.

In the study of species over a perfect field $k$, the action of the absolute Galois group and different forms of the path algebra arise. Since Galois cohomology is crucial to our study of the rational points of quiver moduli spaces over $k$ and leads to twisted quiver representations and different forms of quiver moduli spaces, one can ask what is the relationship between the rational points of these twisted moduli spaces and modules over the tensor algebra of species. Since the results mentioned above for species assume that the quiver has no oriented cycles, but we do not make such an assumption, one may ask whether we can say more about the rational points of quiver moduli spaces over perfect fields if we impose some conditions on the quiver. For example, over the reals, it is relatively easy to construct quaternionic representations of a quiver with loops (or oriented cycles); however, if the quiver does not have any oriented cycles, it is not clear whether the type map is ever surjective.

\begin{quest} 
How is the image of the type map $\cT :\Modgs(k) \lra \Br(k)$ related to the quiver $Q$? For example, over the reals, is the image of the type map trivial when $Q$ has no oriented cycles?
\end{quest}

In \cite{Hubery}, quiver automorphism also arise in the study of hereditary algebras over perfect fields. Over the reals, this should be related to our study in \cite[$\S$4]{HS_quiver_autos} of special submanifolds called branes in hyperk\"{a}hler quiver varieties, which also uses both complex conjugation and quiver automorphisms together.

\subsection{Structure of the paper}

The structure of this paper is as follows. In $\S$\ref{quiver_moduli_sect}, we explain how to construct moduli spaces of representations of a quiver over an arbitrary field $k$ following King \cite{king}, and we examine how (semi)stability behaves under base field extension. In $\S$\ref{Galois_actions_sect}, we study actions of $\Gal_k$ for a perfect field $k$ and give a decomposition of the rational points of $\Modgs$ indexed by the Brauer group, using only elementary considerations from group cohomology. In $\S$\ref{sec_gerbes_and_twisted_reps}, we see the benefit of rephrasing the work of $\S$\ref{Galois_actions_sect} in the more sophisticated language of stacks and gerbes, which gives a quicker and more conceptual way to understand the arithmetic aspects of quiver representations over a field (Theorem \ref{decomp_twisted_reps}). In particular, we interpret our decomposition result using twisted quiver representations and show that moduli spaces of twisted quiver representations give different forms of the moduli space $\Modgs$. Finally, in $\S$\ref{sec_field_of_moduli_and_domain_of_def}, we prove Theorem \ref{field_of_moduli_and_domain_of_def}.

\noindent \textbf{Notation.} For a scheme $S$ over a field $k$ and a field extension $L/k$, we denote by $S_L$ the base change of $S$ to $L$. For a point $s \in S$, we let $\kappa(s)$ denote the residue field of $s$. A quiver $Q=(V,A,h,t)$ is an oriented graph, consisting of a finite vertex set $V$, a finite arrow set $A$, a tail map $t:A\lra V$ and a head map $h:A\lra V$.

\noindent \textbf{Acknowledgements.} We thank the referees of a previous version of this paper, for suggesting that we relate our results to twisted quiver representations and for calling our attention to other arithmetic contexts in which fields of moduli are also studied. V.H. would like to thank Simon Pepin Lehalleur for several very fruitful discussions, which helped turned the former suggestion into what is now $\S$\ref{sec_gerbes_and_twisted_reps}.

\section{Quiver representations over a field}\label{quiver_moduli_sect}

Let $Q=(V,A,h,t)$ be a quiver and let $k$ be a field.

\begin{defn}\label{k_rep_def}
A \textit{representation of $Q$ in the category of $k$-vector spaces} (or \textit{$k$-representation} of $Q$) is a tuple $W:=((W_v)_{v\in V}, (\pphi_a)_{a\in A})$ where:
\begin{itemize}
\item $W_v$ is a finite-dimensional $k$-vector space for all $v\in V$;
\item  $\pphi_a: W_{t(a)}\lra W_{h(a)}$ is a $k$-linear map for all $a\in A$.
\end{itemize} 
\end{defn}

\noindent There are natural notions of morphisms of quiver representations and subrepresentations. The dimension vector of a $k$-representation $W$ is the tuple $d=(\dim_k W_v)_{v\in V}$; we then say $W$ is $d$-dimensional.

\subsection{Slope semistability}\label{slope_semistability}

Following King's construction of moduli spaces of quiver representations over an algebraically closed field \cite{king}, we introduce a stability parameter $\theta:=(\theta_v)_{v\in V} \in \ZZ^V$ and the associated slope function $\mu_\theta$, defined for all non-zero $k$-representations $W$ of $Q$, by $$\mu_\theta(W) := \mu^k_\theta(W) := \frac{\sum_{v\in V} \theta_v\dim_k W_v}{\sum_{v\in V} \dim_k W_v} \in \Q.$$

\begin{defn}
A $k$-representation $W$ of $Q$ is:
\begin{enumerate}
\item $\theta$\textit{-semistable} if $\mu_\theta(W') \leq \mu_\theta(W)$ for all $k$-subrepresentation $0 \neq W'\subset W$.
\item $\theta$\textit{-stable} if $\mu_\theta(W') < \mu_\theta(W)$ for all $k$-subrepresentation $0 \neq W'\subsetneq W$.
\item $\theta$\textit{-polystable} if it is isomorphic to a direct sum of $\theta$-stable representations of equal slope.
\end{enumerate}
\end{defn}
 
\noindent The category of $\theta$-semistable $k$-representations of $Q$ with fixed slope $\mu\in \Q$ is an Abelian, Noetherian and Artinian category, so it admits Jordan-H\"older filtrations. The simple (resp.\ semisimple) objects in this category are precisely the stable (resp.\ polystable) representations of slope $\mu$ (proofs of these facts are readily obtained by adapting the arguments of \cite{Seshadri_S_equiv} to the quiver setting). The graded object associated to any Jordan-H\"older filtration of a semistable representation is  by definition polystable and its isomorphism class as a graded object is independent of the choice of the filtration. Two $\theta$-semistable $k$-representations of $Q$ are called $S$-equivalent if their associated graded objects are isomorphic.

\begin{defn}\label{scss_def}
Let $W$ be a $k$-representation of $k$; then a $k$-subrepresentation $U \subset W$ is said to be \textit{strongly contradicting semistability} (\textit{scss}) with respect to $\theta$ if its slope is maximal among the slopes of all subrepresentations of $W$ and, for any $W' \subset W$ with this property, we have $U\subset W' \Rightarrow U = W'$.
\end{defn}

\noindent For a proof of the existence and uniqueness of the \textit{scss} subrepresentation, we refer to \cite[Lemma 4.4]{Reineke_lectures}. The \textit{scss} subrepresentation satisfies $\Hom(U,W/U)=0$. 
Using the existence and uniqueness of the \textit{scss}, one can inductively construct a unique Harder--Narasimhan filtration; for example, see \cite[Lemma 4.7]{Reineke_lectures}.

We now turn to the study of how the notions of semistability and stability behave under a field extension $L/k$. A $k$-representation $W=((W_v)_{v\in V},(\pphi_a)_{a\in A})$ of $Q$ determines an $L$-representation $L\otimes_k W := ((L\otimes_k W_v)_{v\in V}, (Id_L\otimes \pphi_a)_{a\in A})$ (or simply $L \otimes W)$, where $L\otimes_k W_v$ is equipped with its canonical structure of $L$-vector space and $Id_L\otimes \pphi_a$ is the extension of the $k$-linear map $\pphi_a$ by $L$-linearity. Note that the dimension vector of $L\otimes_k W$ as an $L$-representation is the same as the dimension vector of $W$ as a $k$-representation. We prove that semistability of quiver representations is invariant under base field extension, by following the proof of the analogous statement for sheaves given in \cite[Proposition 3]{Langton} and \cite[Theorem 1.3.7]{Huybrechts_Lehn}.

\begin{prop}\label{sst_and_field_ext}
Let $L/k$ be a field extension and let $W$ be a $k$-representation. For a stability parameter $\theta \in \ZZ^V$, the following statements hold.
\begin{enumerate}
\item If $L\otimes_k W$ is $\theta$-semistable (resp.\ $\theta$-stable) as an $L$-representation, then $W$ is $\theta$-semistable (resp.\ $\theta$-stable) as a $k$-representation.
\item If $W$ is $\theta$-semistable as a $k$-representation, then $L\otimes_k W$ is $\theta$-semistable as an $L$-representation.
\end{enumerate}
Moreover, if $(W^i)_{1\leq i\leq l}$ is the Harder-Narasimhan filtration of $W$, then $(L\otimes_k W^i)_{1\leq i\leq l}$ is the Harder-Narasimhan filtration of $L\otimes_k W$.
\end{prop}
 
 \begin{proof}
Let us suppose that $L\otimes_k W$ is $\theta$-semistable as an $L$-representation. Then, given a $k$-subrepresentation $W'\subset W$, we have $$\mu_\theta^k(W') = \mu^L_\theta(L\otimes_k W') \leq \mu^L_\theta(L\otimes_k W) = \mu_\theta^k(W).$$ Thus, $W$ is necessarily $\theta$-semistable as a $k$-representation; this proof shows that (1) also holds for stability.

First we can reduce (2) to finitely generated extensions $L/k$ as follows. Let $W^L$ be an $L$-subrepresentation of $L\otimes_k W$. For each $v\in V$, choose an $L$-basis $(b^v_j)_{1\leq j\leq d_v}$ of $W^L_v$ and write $b^v_j = \sum_i a^v_{ij} e^v_{ij}$ (a finite sum with $a^v_{ij}\in L$ and $e^v_{ij}\in W_v$). Let $L'$ be the subfield of $L$ generated by the $a^v_{ij}$. The $(e^v_{ij})$ generate an $L'$-subrepresentation $W^{L'}$ of $L' \otimes_k W$ that satisfies $L\otimes_{L'} W^{L'} = W^{L'}$. If $L\otimes_k W$ is not semistable, there exists $W^L\subset L\otimes_k W$ such that $\mu_\theta^L(W^L) > \mu_\theta^L(L\otimes_k W)=\mu_\theta(W)$, then $L'\otimes_k W$ is not semistable, as $\mu^{L'}(W^{L'})>\mu_\theta^{L'}(L'\otimes_k W)$.

By filtering $L/k$ by various subfields, and taking unions into account as well, it suffices to verify the following cases:
\begin{enumerate}
\renewcommand{\labelenumi}{(\roman{enumi})}
\item $L/k$ is a Galois extension;
\item $L/k$ is a separable algebraic extension;
\item $L/k$ is a purely inseparable finite extension;
\item $L/k$ is a purely transcendental extension, of transcendence degree $1$.
\end{enumerate}

For (i), we prove the statement by contrapositive, using the existence and uniqueness of the \textit{scss} $U \subsetneq L\otimes_k W$ with $\mu^L_\theta(U)>\mu^L_\theta(L\otimes_k W)$. 
For $\tau\in\Aut(L/k)$, we construct an $L$-subrepresentation $\tau(U)$ of $L \otimes_k W$ of the same dimension vector and slope as $U$ as follows. For each $v\in V$, the $k$-automorphism $\tau$ of $L$ induces an $L$-semilinear transformation of $L\otimes_k W_v$ (i.e., an additive map satisfying $\tau(zw) = \tau(z)\tau(w)$ for all $z\in L$ and all $w\in W_v$), which implies that $\tau(U_v)$ is an $L$-vector subspace of $L\otimes_k W_v$, and the map $\tau\pphi_a\tau^{-1}: \tau(U_{t(a)}) \lra \tau(U_{h(a)})$ is $L$-linear. 
By uniqueness of the \textit{scss} subrepresentation $U$, we must have $\tau(U)=U$ for all $\tau\in \Aut(L/k)$. Moreover, for all $v\in V$, the $k$-vector space $U_v^{\Aut(L/k)}$ is a subspace of $(L\otimes_k W_v)^{\Aut(L/k)}=W_v$, as $L/k$ is Galois. Then $U^{\Aut(L/k)}\subset W$ is a  $k$-subrepresentation with $\mu_\theta(U^{\Aut(L/k)}) = \mu^L_\theta(U)$; thus $W$ is not semistable.

For (ii), we pick a Galois extension $N$ of $k$ containing $L$; then we conclude the claim using (i) and (1).

For (iii), by Jacobson descent, an $L$-subrepresentation $W^L\subset L\otimes_k W$ descends to a $k$-subrepresentation of $W$ if and only if $W^L$ is invariant under the algebra of $k$-derivations of $L$, which is the case for the \textit{scss} $L$-subrepresentation $U\subset L\otimes_k W$. Indeed, let us consider a derivation $\delta\in \mathrm{Der}_k(L)$ and, for all $v\in V$, the induced transformation $(\psi_{\delta})_v:= (\delta\otimes_k \mathrm{Id}_{W_v}): L\otimes_k W_v \lra L\otimes_k W_v$. Then, for all $v\in V$, all $\lambda\in L$ and all $u\in U_v$, one has $(\psi_\delta)_v(\lambda u) = \delta(\lambda)u + \lambda\psi_\delta(u)$. As the composition $$\ov{(\psi_\delta)_v}: U_v \hookrightarrow L\otimes_k W_v \overset{(\psi_\delta)_v}{\lra} L\otimes_k W_v \lra (L\otimes_k W_v)/U_v$$ is $L$-linear, we obtain a morphism of $L$-representations $\ov{\psi_\delta}:U\lra (L\otimes_k W)/U$, which must be zero as $U$ is the \textit{scss} subrepresentation of $L\otimes_k W$. As $U$ is invariant under $\psi_{\delta}$, it descends to a $k$-subrepresentation of $W$; then we argue as in (i). 

For (iv), we distinguish two cases. If $k$ is infinite, the fixed subfield of $k(X)$ for the action of $\Aut(k(X)/k) \simeq \mathbf{PGL}(2;k)$ is $k$ (\textit{cf.}\  \cite[p.\ 254]{Robert}), so we can argue as in (i). If $k$ is finite, the fixed subfield of $k(X)$ for the action of $\Aut(k(X)/k)$ is strictly larger than $k$, as $\mathbf{PGL}(2;k)$ is finite. Let then $\kb$ be an algebraic closure of $k$. If $W$ is semistable, so is $\kb\otimes_k W$ in view of the above, since $\kb/k$ is algebraic. As $\kb$ is infinite, $\kb(X) \otimes_{\kb} (\kb\otimes_k W) =  \kb(X)\otimes_k W$ is also semistable. Since  $\kb(X)\otimes_k W = \kb(X) \otimes_{k(X)} (k(X)\otimes_k W)$, we can conclude that $k(X)\otimes_k W$ is a semi\-stable $k(X)$-representation by Part (1).
 \end{proof}
 
\begin{rmk}\label{stability_and_field_ext} Part (2) of Proposition \ref{sst_and_field_ext} is not true if we replace semistability by stability, as is evident if we set $k=\R$ and $L=\C$: for a $\theta$-stable $\R$-representation $W$, its complexification $\C\otimes W$ is a $\theta$-semistable $\C$-representation by Proposition \ref{sst_and_field_ext} and either, for all $\C$-subrepresentations $U\subset \C\otimes W$, one has $\mu^{\C}_\theta(U) < \mu^\C_\theta(\C\otimes W)$, in which case $\C\otimes W$ is actually $\theta$-stable as a $\C$-representation; or there exists a $\C$-subrepresentation $U\subset L\otimes W$ such that $\mu^\C_\theta(U)=\mu^\C_\theta(\C\otimes W)$. In the second case, let $\tau(U)$ be the $\C$-subrepresentation of $\C\otimes W$ obtained by applying the non-trivial element of $\Aut(\C/\R)$ to $U$. Note that $\tau(U)\neq U$, as otherwise it would contradict the $\theta$-stability of $W$ as an $\R$-representation (as in the proof of Part (2) of Proposition \ref{sst_and_field_ext}). It is then not difficult, adapting the arguments of \cite{Ramanan_hyperelliptic,Sch_JSG}, to show that $U$ is a $\theta$-stable $\C$-representation and that $\C\otimes W \simeq U\oplus \tau(U)$; thus $\C\otimes W$ is only $\theta$-polystable as a $\CC$-representation.
\end{rmk}

This observation motivates the following definition.

\begin{defn}\label{geom_stability_def}
A $k$-representation $W$ is $\theta$\textit{-geometrically stable} if $L\otimes_k W$ is $\theta$-stable as an $L$-representation for all extensions $L/k$.
\end{defn}

\noindent Evidently, the notion of geometric stability is invariant under field extension. In fact, if $k=\kb$, then being geometrically stable is the same as being stable: this can be proved directly, as in \cite[Corollary 1.5.11]{Huybrechts_Lehn}, or as a consequence of Proposition \ref{GIT_charac_of_sst_and_geom_st} below. This implies that a $k$-representation $W$ is $\theta$-geometrically stable if and only if $\kb\otimes_k W$ is $\theta$-stable (the proof is the same as in Part (2) - Case (iv) of Proposition \ref{sst_and_field_ext}).

\subsection{Families of quiver representations}\label{moduli_functors_section}

A family of $k$-representations of $Q$ parametrised by a $k$-scheme $B$ is a representation of $Q$ in the category of vector bundles over $B$, denoted $\cE = ((\cE_v)_{v \in V}, (\varphi_a)_{a \in A}) \lra B$. For  $d=(d_v)_{v\in V} \in \NN^V$, we say a family $\cE\lra B$ is $d$-dimensional if, for all $v\in V$, the rank of $\cE_v$ is $d_v$. For a morphism of $k$-schemes  $f:B'\lra B$, there is a pullback family $f ^*\cE:= (f^*\cE_v)_{v\in V}$ over $B'$. For $b\in B$ with residue field $\kappa(b)$, we let $\cE_b$ denote the $\kappa(b)$-representation obtained by pulling back $\cE$ along $u_b:\spec \kappa(b) \lra B$.

\begin{defn}\label{sst_in_families}
A family $\cE \lra B$ of $k$-representations of $Q$ is called:
\begin{enumerate}
\item $\theta$\textit{-semistable} if, for all $b\in B$, the $\kappa(b)$-representation $\cE_{b}$ is $\theta$-semistable.
\item $\theta$\textit{-geometrically stable} if, for all $b\in B$, the $\kappa(b)$-representation $\cE_{b}$ is $\theta$-geometrically stable.
\end{enumerate} 
\end{defn}

\noindent For a family $\cE \lra B$ of $k$-representations of $Q$, the subset of points $b \in B$  for which $\cE_b$ is $\theta$-semistable (resp.\ $\theta$-geometrically stable) is open; one can prove this by adapting the argument in \cite[Proposition 2.3.1]{Huybrechts_Lehn}. By Proposition \ref{sst_and_field_ext} and Definition \ref{geom_stability_def}, the pullback of a $\theta$-semistable (resp.\ $\theta$-geometrically stable) family is semistable (resp.\ geometrically stable). Therefore, we can introduce the following moduli functors:
\begin{equation}\label{moduli_functors}
F_{Q,d}^{\theta-ss}: (Sch_k)^\op  \lra  \Sets \quad \text{and} \quad F_{Q,d}^{\theta-gs}: (Sch_k)^\op  \lra  \Sets,
\end{equation}
where $(Sch_k)^\op$ denotes the opposite category of the category of $k$-schemes and, for $B \in Sch_k$, we have that $F_{Q,d}^{\theta-ss}(B)$ (resp.\ $F_{Q,d}^{\theta-gs}(B)$) is the set of isomorphism classes of $\theta$-semistable (resp.\ $\theta$-geometrically stable) $d$-dimensional families over $B$ of $k$-representations of $Q$.

We follow the convention in that a scheme $\mathcal{M}$ is a \textit{coarse moduli space} for the moduli functor $F : (Sch_k)^\op \lra \Sets$ is $F$ if it comes equipped with a universal natural transformation $F \lra \Hom(-,\mathcal{M})$ inducing bijections $\mathcal{M}(\Omega) \simeq F(\Omega)$ for all algebraically closed fields $\Omega$. When referring to the first condition only, it will sometimes be convenient to say that $\mathcal{M}$ \textit{corepresents} the functor $F$.

\subsection{The GIT construction of the moduli space}\label{GIT_constr_section}

Fix a field $k$ and dimension vector $d=(d_v)_{v\in V}\in \N^V$; then every $d$-dimensional $k$-representation of $Q$ is isomorphic to a point of the following affine space over $k$
\[ \Rep := \prod_{a\in A} \Mat_{d_{h(a)}\times d_{t(a)}}.\]
The reductive group $\G:=\prod_{v\in V} \GL_{d_v}$ over $k$ acts algebraically on $\Rep$ by conjugation: for $g = (g_v)_{v \in V} \in \G$ and $M = (M_a)_{a \in A} \in \Rep$, we have
 \begin{equation}\label{action_of_G_on_Rep}
g\cdot M := (g_{h(a)} M_a g_{t(a)}^{-1})_{a\in A}.
\end{equation}

There is a tautological family $\cF \lra \Rep$ of $d$-dimensional $k$-representations of $Q$, where $\cF_v$ is the trivial rank $d_v$ vector bundle on $\Rep$. 

\begin{lemma}\label{local_univ_family}
The tautological family $\cF \lra \Rep$ has the local universal property; that is, for every family $\cE=((\cE_v)_{v \in V}, (\varphi_a)_{a \in A}) \lra B$ of representations of $Q$ over a $k$-scheme $B$, there is an open covering $B = \cup_{i \in I} B_i$ and morphisms $f_i : B_i \lra \Rep$ such that $\cE|_{B_i} \cong f_i^*\cF$.
\end{lemma}
\begin{proof}
Take an open cover of $B$ on which all the (finitely many) vector bundles $\cE_v$ are trivialisable, then the morphisms $f_i$ are determined by the morphisms $\varphi_a$.
\end{proof}

We will construct a quotient of the $\G$-action on $\Rep$ via geometric invariant theory (GIT) using a linearisation of the action by a stability parameter $\theta=(\theta_v)_{v\in V}\in \Z^V$. Let us set $\theta':=(\theta'_v)_{v\in V}$ where $\theta'_v := \theta_v \sum_{\alpha\in V} d_\alpha - \sum_{\alpha\in V} \theta_\alpha d_\alpha$  for all $v\in V$; then one can easily check that $\theta'$-(semi)stability is equivalent to $\theta$-(semi)stability. We define a character $\chi_\theta: \G \lra \GG_m$ by
\begin{equation}\label{the_character}
\chi_{\theta}( (g_v)_{v \in V} ):=\prod_{v\in V} (\det g_v)^{-\theta'_v}.
\end{equation} Any such character $\chi: \G \lra \Gm$ defines a lifting of the $\G$-action on $\Rep$ to the trivial line bundle $\Rep\times \A^1$, where $\G$ acts on $\AA^1$ via multiplication by $\chi$. As the subgroup $\Delta \subset \G$, whose set of $R$-points (for $R$ a $k$-algebra) is \begin{equation}\label{Delta_subgp}
\Delta(R) := \{ (t I_{d_v})_{v\in V} : t\in R^\times\} \cong \GG_m(R),
\end{equation} acts trivially on $\Rep$, invariant sections only exist if $\chi^{(R)}(\Delta(R)) = \{1_{R^\times}\}$ for all $R$; this holds for $\chi_\theta$, as $\sum_{v\in V} \theta'_v d_v =0$. Let $\cL_\theta$ denote the line bundle $\Rep\times\A^1$ endowed with the $\G$-action induced by $\chi_\theta$ and by $\cL_\theta^n$ its $n$-th tensor power for $n\geq 1$ (endowed with the action of $\chi_\theta^n$). The invariant sections of $\cL_\theta^n$ are $\chi_\theta$-semi-invariant functions; that is, morphisms $f:\Rep\lra\A^1$ satisfying  $f(g\cdot M) = \chi_\theta(g)^n\, f(M)$ for all $g\in\G$ and all $M\in\Rep$.
\begin{defn}
A point $M\in \Rep$ is called:
\begin{enumerate}
\item $\chi_\theta$\textit{-semistable} if there exists an integer $n>0$ and a $\G$-invariant section $f$ of $\cL_\theta^n$ such that $f(M)\neq 0$.
\item $\chi_\theta$\textit{-stable} if there exists an integer $n>0 $ and a $\G$-invariant section $f$ of $\cL_\theta^n$ such that $f(M)\neq 0$, the action of $\G$ on $(\Rep)_f$ is closed and $\dim_{\kappa(M)} (\Stab(M)/\Delta_{\kappa(M)}) = 0$, where $\Stab(M) \subset \mathbf{G}_{Q,d,\kappa(M)}$ is the stabiliser group scheme of $M$.
\end{enumerate}
We denote the set of $\chi_\theta$-(semi)stable points in $\Rep$ by $\Rep^{\chi_\theta-(s)s}$.
\end{defn} 
\noindent Evidently, $\Rep^{\chi_\theta-ss}$ and $\Rep^{\chi_\theta-s}$ are $\G$-invariant open subsets. Moreover, these subsets commute with base change (\textit{cf.}\ \cite[Proposition 1.14]{Mumford} and \cite[Lemma 2]{Seshadri_GIT}). Mumford’s GIT (or, more precisely, Seshadri’s extension of GIT \cite{Seshadri_GIT}) provides a categorical and good quotient of the $\G$-action on $\Rep^{\chi_\theta-ss}$
\[ \pi : \Rep^{\chi_\theta-ss} \lra \Rep/\!/_{\chi_\theta} \G:= \proj \bigoplus_{n \geq 0} H^0(\Rep, \cL_\theta^{n})^{\G}, \]
which restricts to a geometric quotient $\pi|_{\Rep^{\chi_\theta-s}} : \Rep^{\chi_\theta-s} \lra \Rep^{\chi_\theta-s}/\G$.

Given a geometric point $M : \spec \Omega \lra \Rep$, let us denote by $\Lambda(M)$ the set of $1$-parameter subgroups $\lambda:\mathbb{G}_{m,\Omega} \lra \mathbf{G}_{Q,d,\Omega}$ such that the morphism $\GG_{m,\Omega} \lra \rep_{Q,d,\Omega}$, given by the $\lambda$-action on $M$, extends to $\AA^1_\Omega$. As $\Rep$ is separated, if this morphism extends, its extension is unique. If $M_0$ denotes the image of $0 \in \AA^1_\Omega$, the weight of the induced action of $\mathbb{G}_{m,\Omega}$ on $\cL_{\theta,\Omega}|_{M_0}$ is $(\chi_{\theta,\Omega},\lambda)\in \Z$, where $(-,-)$ denotes the natural pairing of characters and 1-parameter subgroups.

\begin{prop}[Hilbert-Mumford criterion \cite{king}]\label{HM_criterion}
For a geometric point $M: \spec \Omega \ra \Rep$, we have
\begin{enumerate}
\item $M$ is $\chi_\theta$-semistable if and only if  $(\chi_{\theta,\Omega},\lambda) \geq 0$ for all $\lambda\in\Lambda(M)$;
\item $M$ is $\chi_\theta$-stable if and only if $(\chi_{\theta,\Omega},\lambda) \geq 0$ for all $\lambda\in\Lambda(M)$, and $(\chi_{\theta,\Omega},\lambda)= 0$ implies $\im\lambda \subset \Stab(M)$, where $\Stab(M) \subset \mathbf{G}_{Q,d,\Omega}$ is the stabiliser group scheme of $M$.
\end{enumerate}
\end{prop}

\begin{proof}
If $k$ is algebraically closed and $\Omega=k$, this is \cite[Proposition 2.5]{king}; then the above result follows as GIT (semi)stability commutes with base change.
\end{proof}

Before we relate slope (semi)stability and GIT (semi)stability for quiver representations, let $\Rep^{\theta-ss}$ (resp.\ $\Rep^{\theta-gs}$) be the open subset of points in $\Rep$ over which the tautological family $\cF$ is $\theta$-semistable (resp.\ $\theta$-geometrically stable).

\begin{prop}\label{GIT_charac_of_sst_and_geom_st}
For $\theta \in \ZZ^V$, we have the following equalities of $k$-schemes:
\begin{enumerate}
\item $\Rep^{\theta-ss} = \Rep^{\chi_\theta-ss}$;
\item $\Rep^{\theta-gs} = \Rep^{\chi_\theta-s}$.
\end{enumerate}
\end{prop}

\begin{proof}
Since all of these $k$-subschemes of $\Rep$ are open, it suffices to verify these equalities on $\ov{k}$-points, for which one uses \cite[Proposition 3.1]{king} (we note that we use the opposite inequality to King in our definition of slope (semi)stability, but this is rectified by the minus sign appearing in \eqref{the_character} for the definition of $\chi_\theta$).
\end{proof}

\noindent Proposition \ref{GIT_charac_of_sst_and_geom_st} readily implies the result claimed at the end of $\S$\ref{slope_semistability}, which we state here for future reference.

\begin{cor}\label{charac_of_geom_stability}
A $k$-representation $W$ is $\theta$-geometrically stable if and only if $\kb\otimes_k W$ is $\theta$-stable. In particular, if $k=\kb$, then $\theta$-geometric stability is equivalent to $\theta$-stability. 
\end{cor}

\noindent Finally, we show the existence of coarse moduli spaces of $\theta$-semistable (resp.\ $\theta$-geometrically stable) $k$-representations of $Q$ for an arbitrary field $k$. For an algebraically closed field $k$, this result is proved in \cite[Proposition 5.2]{king}. 

\begin{thm}\label{GIT_const_of_ModQd}
The $k$-variety $\Mod:=\Rep /\!/_{\chi_\theta} \G$ is a coarse moduli space for the functor $\Fss$ and the natural map $\Fss(\kb) \lra \Mod(\kb)$ is surjective. Moreover, $\Modgs := \Rep^{\chi_\theta-s}/\G$ is an open $k$-subvariety of $\Mod$ which is a coarse moduli space for the functor $\Fgs$ and the natural map $\Fgs(\kb) \lra \Modgs(\kb)$ is bijective.
\end{thm}

\begin{proof}
First, we verify that $\Mod$ is a $k$-variety: it is of finite type over $k$, as the ring of sections of powers of $\cL_\theta$ that are invariant for the reductive group $\G$ is finitely generated. 
Moreover, $\Mod$ is separated, as it is projective over the affine $k$-scheme $\spec \cO(\Rep)^{\G}$. Finally $\Mod$ is integral, as $\Rep^{\chi_\theta-ss}$ is and this property is inherited by the categorical quotient. 

Since the tautological family $\cF^{\theta-ss} \lra \Rep^{\theta-ss}$ has the local universal property by Lemma \ref{local_univ_family} and also the $\G$-action on $\Rep$ is such that  $M,M' \in \Rep$ lie in the same $\G$-orbit if and only if $\cF_M \cong \cF_{M'}$, it follows that any $\G$-invariant morphism $p: \Rep^{\theta-ss} \lra Y$ is equivalent to a natural transformation $\eta_p : \Fss \lra \Hom(-,Y)$ (\textit{cf.}\  \cite[Proposition 2.13]{Newstead_TATA}). 
As $\Rep^{\theta-ss} = \Rep^{\chi_\theta-ss}$ by Proposition \ref{GIT_charac_of_sst_and_geom_st}, and as $\pi : \Rep^{\chi_\theta-ss} \lra \Rep/\!/_{\chi_\theta} \G=\Mod$ is a universal $\G$-invariant morphism, it follows that $\Mod$ corepresents $\Fss$, and similarly $\Modgs$ corepresents $\Fgs$.

The points of $\Mod(\kb)$ are in bijection with equivalence classes of $\G(\kb)$-orbits of $\chith$-semistable $\kb$-points, where $\kb$-points $M_1$ and $ M_2$ are equivalent if their orbit closures intersect in $\Rep^{\chi_\theta-ss}(\kb)$  (\textit{cf.}\ \cite[Theorem 4]{Seshadri_GIT}). By \cite[Proposition 3.2.(ii)]{king}, this is the same as the $S$-equivalence of $\cF_{M_1}$ and $\cF_{M_2}$ as $\theta$-semistable $\kb$-representation of $Q$; hence the surjectivity of the natural map $\Fss(\kb)\lra \Mod(\kb)$. Likewise, $\Mod(\kb)$ is in bijection with the set of $\G(\kb)$-orbits of $\chith$-stable $\kb$-points of $\Rep$, which, by \cite[Proposition 3.1]{king}, is in bijection with the set of $\theta$-stable $d$-dimensional $\kb$-representations of $Q$.
\end{proof}

We end this section with a result that is used repeatedly in Sections \ref{Galois_actions_sect}.

\begin{cor}\label{stabiliser_of_GIT_stable_points}
For $M\in\Reps$, we have $\Stab(M) =\Delta_{\ka(M)} \subset \mathbf{G}_{Q,d,\ka(M)}$.
\end{cor}

\begin{proof}
$\Stab(M)\subset \mathbf{G}_{Q,d,\ka(M)}$ is isomorphic to $\Aut(\cF_M)$, where $\cF\lra\Rep$ is the tautological family, and $\cF_M$ is $\theta$-geometrically stable. The endomorphism group of a stable $k$-representation of $Q$ is a finite dimensional division algebra over $k$ (\textit{cf.}\ \cite[Proposition 1.2.8]{Huybrechts_Lehn}). Let $\ov{\ka(M)}$ be an algebraic closure of $\ka(M)$; then, as $\ov{\ka(M)}\otimes\cF_M  $ is $\theta$-stable and $\ov{\ka(M)}$ is algebraically closed, $\End(\ov{\ka(M)}\otimes\cF_M  ) = \ov{\ka(M)}$. Since $\ov{\ka(M)}\otimes\End(\cF_M)  \subset \End(\ov{\ka(M)}\otimes\cF_M)$, it follows that $\End(\cF_M) = \ka(M)$ and thus $\Aut(\cF_M)  \simeq \Delta_{\ka(M)}$.
\end{proof}

\section{Rational points of the moduli space}\label{Galois_actions_sect}

Throughout this section, we assume $k$ is a perfect field and we fix an algebraic closure $\ov{k}$ of $k$. For a $k$-scheme $X$, there is a left action of the Galois group $\Gal_k:= \Gal(\ov{k}/k)$ on the set of $\ov{k}$-points $X(\ov{k})$ as follows: for $\tau \in \Gal_k$ and $x : \spec \ov{k} \lra X$, we let $\tau \cdot x:= x \circ \tau^*$, where $\tau^* : \spec \ov{k} \lra  \spec \ov{k}$ is the morphism of $k$-schemes induced by the $k$-algebra homomorphism $\tau:\ov{k}\lra\ov{k}$. As $k$ is perfect, $X(k) = X(\ov{k})^{\Gal_k}$, where the right side denotes the fixed-point set the $\Gal_k$-action on $X(\ov{k})$. If $X_{\ov{k}}=\spec\kb\times_{\spec k} X$, then $X_{\ov{k}}(\ov{k}) = X(\ov{k})$ and $\Gal_k$ acts on $X_{\ov{k}}$ by $k$-scheme automorphisms and, as $k$ is perfect, we can recover $X$ as $X_{\ov{k}}/\Gal_k$.

\subsection{Rational points arising from rational representations}\label{rational_pts_section}

\ The moduli space $\Mod$ constructed in Section \ref{quiver_moduli_sect} is a $k$-variety, so the Galois group $\Gal_k:= \Gal(\ov{k}/k)$ acts on $\Mod(\ov{k})$ as described above and the fixed points of this action are the $k$-rational points. Alternatively, we can describe this action using the presentation of $\Mod$ as the GIT quotient $\Rep/\!/_{\chi_\theta}\G$. The $\Gal_k$-action on $\Rep(\ov{k}) = \prod_{a\in A} \Mat_{d_{h(a)} \times d_{t(a)}}(\ov{k})$ and $\G(\ov{k}) = \prod_{v\in V} \GL_{d_v}(\ov{k})$ is given by applying a $k$-automorphism $\tau\in\Gal_k=\Aut(\ov{k}/k)$ to the entries of the matrices $(M_a)_{a\in A}$ and $(g_v)_{v\in V}$. Both actions are by homeomorphisms in the Zariski topology and the second action is by group automorphisms and preserves the subgroup $\Delta(\ov{k})$ defined in \eqref{Delta_subgp}. We denote these actions as follows 
\begin{equation}\label{Gal_action_on_Rep}
\Phi:   \Gal_k \times \Rep(\ov{k})  \lra \rep(\ov{k}), \quad \big(\tau,(M_a)_{a\in A}\big)  \lmt  \big(\tau(M_a)\big)_{a\in A}
\end{equation}
and 
\begin{equation}\label{Gal_action_on_G}
\Psi: \Gal_k \times \G(\ov{k})  \lra  \G(\ov{k}),\quad \big(\tau,(g_v)_{v\in V}\big)  \lmt  \big(\tau(g_v)\big)_{v\in V}.
\end{equation} They satisfy the following compatibility relation with the action of $\G(\ov{k})$ on $\Rep(\ov{k})$: for all $g\in\G(\ov{k})$, all $M\in\rep(\kb)$ and all $\tau\in\Gal_k$, one has 
\begin{equation}\label{comp_btw_the_actions_Gal_case}
\Phi_\tau(g\cdot M) = \Psi_\tau(g)\cdot \Phi_\tau(M)
\end{equation} (i.e., the $\G(\kb)$-action on $\Rep(\kb)$ extends to an action of $\G(\kb)\rtimes\Gal_k$). For convenience, we will often simply denote $\Phi_\tau(M)$ by $\tau(M)$ and $\Psi_\tau(g)$ by $\tau(g)$.


\begin{prop}\label{Gal_preserves_GIT_sst}
The $\Gal_k$-action on $\Rep(\kb)$ preserves $\Rep^{\chi_\theta-(s)s}(\kb)$. Moreover, if $M_1, M_2$ are two GIT-semistable points whose $\G(\kb)$-orbits closures meet in $\Repss(\kb)$, then, for all $\tau\in\Gal_k$, the same is true for $\tau(M_1)$ and $\tau(M_2)$.
\end{prop}

\begin{proof}
The first statement holds, as the $\Gal_k$-action preserves the $\chi_\theta$-semi-invariant functions due to the compatibility relation \eqref{comp_btw_the_actions_Gal_case}, and moreover, for $M\in\Rep(\kb)$ and $\tau\in\Gal_k$, we have $\Stab_{\G(\kb)} \big(\tau(M)\big) = \tau\big(\Stab_{\G(\kb)}(M)\big)$. The second statement follows from  \eqref{comp_btw_the_actions_Gal_case} and the continuity of $\tau$ in the Zariski topology of $\Repss(\kb)$.
\end{proof}

\noindent Proposition \ref{Gal_preserves_GIT_sst} combined with the compatibility relation \eqref{comp_btw_the_actions_Gal_case} readily implies that $\Gal_k$ acts on the set of $\kb$-points of the $k$-varieties $\Mod=\Rep/\!/_{\chi_\theta}\G$ and $\Modgs = \Reps/\G$. Explicitly, the $\Gal_k$-action on the orbit space $\Reps(\kb)/\G(\kb)$ is given by
\begin{equation}\label{induced_Gal_k_action_on_geom_quotient}
(\G(\kb)\cdot M ) \lmt (\G(\kb)\cdot \tau(M)).
\end{equation}

\noindent Since $k$ is assumed to be a perfect field, this $\Gal_k$-action on the $\kb$-varieties $\Mod(\kb)$ and $\Modgs(\kb)$ suffices to recover the $k$-schemes $\Mod$ and $\Modgs$. In particular, the $\Gal_k$-actions just described on $\Mod(\kb)$ and $\Modgs(\kb)$ coincide with the ones described algebraically at the beginning of the present section.

\begin{rmk} We can intrinsically define the Galois action on $\Rep(\kb)$ by defining a $\Gal_k$-action on arbitrary $\kb$-representations of $Q$ as follows. If $W=((W_v)_{v\in V},$ $(\pphi_a)_{a\in A})$ is a $\kb$-representation of $Q$, then, for $\tau\in \Gal_k$, we define $W^\tau$ to be the representation $(W_v^\tau, v\in V; \phi_a^\tau;a\in A)$ where:
\begin{itemize}
\item $W_v^\tau$ is the $\kb$-vector space whose underlying Abelian group coincides with that of $W_v$ and whose external multiplication is given by $\lambda\cdot_{\tau} w := \tau^{-1}(\lambda)w$ for $\lambda\in \kb$ and $w\in W_v$. 
\item The map $\phi_a^\tau$ coincides with $\phi_a$, which is $\kb$-linear for the new $\kb$-vector space structures, as $\phi_a^\tau(\lambda\cdot_\tau w)= \phi_a(\tau^{-1}(\lambda)w) = \tau^{-1}(\lambda)\phi_a(w) = \lambda\cdot_\tau \phi_a^\tau(w)$.
\end{itemize}
If $\rho:W'\lra W$ is a morphism of $\kb$-representations and $\tau\in\Gal_k$, we denote by $\rho^\tau:(W')^\tau \lra W^\tau$ the induced homomorphism (which set-theoretically coincides with $\rho$). With these conventions, we have a right action, as $W^{\tau_1\tau_2}=(W^{\tau_1})^{\tau_2}$. Moreover, if we fix a $\kb$-basis of each $W_v$, the matrix of $\phi_a^\tau$ is $\tau(M_a)$, where $M_a$ is the matrix of $\phi_a$, so we recover the $\Gal_k$-action \eqref{Gal_action_on_Rep}. We note that the construction $W\lmt W^\tau$ is compatible with semistability and $S$-equivalence, thus showing in an intrinsic manner that $\Gal_k$ acts on the set of $S$-equivalence classes of semistable $d$-dimensional representations of $Q$.
\end{rmk}

By definition of the moduli spaces $\Mod$ and $\Modgs$, we have natural maps 
\begin{equation}\label{nat_maps_on_k_points}
\Fss(k) \lra \Mod(k) \quad \mathrm{and} \quad \Fgs(k) \lra \Modgs(k),
\end{equation} 
where $\Fss$ and $\Fgs$ are the moduli functors defined at \eqref{moduli_functors}. 
As $k$ is perfect, $\Mod(k) = \Mod(\kb)^{\Gal_k}$ and $\Modgs(k) = \Modgs(\kb)^{\Gal_k}.$ 
The goal of the present section is to use this basic fact in order to understand the natural maps \eqref{nat_maps_on_k_points}. As a matter of fact, our techniques will only apply to $\Fgs(k)\lra \Modgs(k)$, because $\Modgs(\kb)$ is the orbit space $\Reps(\kb) / \G(\kb)$ and all GIT-stable points in $\Rep(\kb)$ have the Abelian group $\Delta(\kb)\simeq\kb^\times$ as their stabiliser for the $\G(\kb)$-action.

Note first that, by definition of the functor $\Fgs$, we have $$\Fgs(k) \simeq \Reps(k) / \G(k),$$ so the natural map $\Fgs(k) \lra \Modgs(k)$ may be viewed as the map 
\begin{equation*}\label{def_of_fGalk}
\fGalk : \Reps(k) / \G(k)  \lra  \big(\Reps(\kb)/\G(\kb)\big)^{\Gal_k},\quad
\G(k)\cdot M  \lmt  \G(\kb)\cdot (\kb \otimes_k M).
\end{equation*} 

\begin{prop}\label{fibres_of_fGalk}
The natural map $\Fgs(k) \lra \Modgs(k)$ is injective.
\end{prop}

\begin{proof} To prove this result, we identify this map with $\fGalk$ and we will show that the non-empty fibres of $\fGalk$ are in bijection with the pointed set 
$$\ker \big(H^1(\Gal_k;\Delta(\kb)) \lra H^1(\Gal_k;\G(\kb))\big)$$ 
where this map is induced by the inclusion $\Delta(\kb)\subset\G(\kb)$. Then the result follows from this claim, as $H^1(\Gal_k;\Delta(\kb)) = \{1\}$ by Hilbert's 90th Theorem (for example, see \cite[Proposition X.1.2]{Serre_local_fields}, where technically the statement is only shown for finite Galois extensions, but the general case follows as $H^1(\varprojlim \Gal(L/k);\, \cdot\,) \simeq \varprojlim H^1(\Gal(L/k);\,\cdot\,)$, where the projective limit is taken over finite Galois sub-extensions $L/k$). It remains to prove the above claim about the fibres of $\fGalk$. For this we consider $M_1, M_2$ in $\Reps(\kb)^{\Gal_k}$ such that $\G(\kb)\cdot M_1 = \G(\kb)\cdot M_2$. Then there exists $g\in\G(\kb)$ such that $g\cdot M_2 = M_1$. Therefore, for all $\tau\in\Gal_k$, we have $$g^{-1}\cdot M_1 = M_2 = \tau(M_2) = \tau(g^{-1}\cdot M_1) = \tau(g^{-1}) \cdot \tau(M_1),$$ so $g\tau(g^{-1}) \in \Stab_{\G(\kb)}(M_1) = \Delta(\kb)$. It is straight-forward to check that the map $$\beta_{M_1,M_2}: \Gal_k \lra \Delta(\kb),\quad
\tau  \lmt  g\tau(g^{-1})
$$ is a normalised $\Delta(\kb)$-valued $1$-cocycle whose cohomology class only depends on the $\G(\kb)^{\Gal_k}$-orbits of $M_1$ and $M_2$. Thus the cohomology class $[\beta_{M_1,M_2}]$ lies in the kernel of the pointed map $H^1(\Gal_k;\Delta(\kb)) \lra H^1(\Gal_k;\G(\kb))$. Hence, for $[M_1]:=\G(k) \cdot M_1 \in \Reps(k)/\G(k)$, there is a map
$$ \fGalk^{-1}(\fGalk([M_1])) \lra \ker\big(H^1(\Gal_k;\Delta(\kb)) \lra H^1(\Gal_k,\G(\kb))\big)$$
sending $[M_2]$ to $\beta_{M_1,M_2}$. We claim this map is bijective. To prove surjectivity, suppose we have a 1-cocycle $\gamma(\tau)=g\tau(g^{-1})\in\Delta(\kb)$ that splits over $\G(\kb)$; then $\tau(g^{-1}\cdot M_1) = g^{-1}\cdot M_1$, since $\Delta(\kb)$ acts trivially on $M_1$, so the cocycle $\beta$ defined using $M_1$ and $M_2:= g^{-1} \cdot M_1$ as above is equal to $\gamma$. To prove that the above map is injective, suppose that the $\Delta(\kb)$-valued $1$-cocycle $\beta$ associated to $M_1$ and $M_2 :=g^{-1} \cdot M_1$ splits over $\Delta(\kb)$ (i.e.\ that there exists $a\in \Delta(\kb)$ such that $g\tau(g^{-1}) =a\tau(a^{-1})$ for all $\tau\in\Gal_k$). Then, on the one hand, $a^{-1}g\in \G(\ov{k})^{\Gal_k}$, as $\tau(a^{-1}g) = a^{-1}g$ for all $\tau\in\Gal_k$, and, on the other hand, 
$$(a^{-1}g)^{-1}\cdot M_1= g^{-1}\cdot(a^{-1}\cdot M_1) = g^{-1}\cdot M_1 = M_2,$$ as $\Delta(\kb)$ acts trivially on $\Rep(\kb)$. Therefore, $\G(k) \cdot M_1 = \G(k) \cdot M_2$.
\end{proof}

In order to study the image of the natural map $$\fGalk: \Fgs(k) \lra \Modgs(k)$$ we introduce a map $\cT$ called the type map, from $\Modgs(\kb)^{\Gal_k}$ to the Brauer group of $k$, denoted by $\Br(k)$: 
\begin{equation}\label{type_map_def}
\cT: \Modgs(k) \lra H^2(\Gal_k; \ov{k}^\times) \cong \Br(k),
\end{equation} which is defined as follows. Consider an orbit $$(\G(\kb)\cdot M) \in \Modgs(k) = (\Reps(\kb)/\G(\kb))^{\Gal_k},$$ of which a representative $M$ has been chosen. As this orbit is preserved by the $\Gal_k$-action, we have that, for all $\tau\in\Gal_k$, there is an element $u_\tau\in \G(\kb)$ such that $u_\tau\cdot \tau(M) = M$. Note that for $\tau=1_{\Gal_k}$, we can simply take $u_\tau=1_{\G(\kb)}$, which we will. Since $(\tau_1\tau_2)(M) = \tau_1(\tau_2(M))$, it follows from the compatibility relation \eqref{comp_btw_the_actions_Gal_case} that, $$u_{\tau_1\tau_2}^{-1}\cdot M = \tau_1(u_{\tau_2}^{-1}\cdot M) = \tau_1(u_{\tau_2}^{-1}) \cdot \tau_1(M) = \tau(u_{\tau_2}^{-1}) u_{\tau_1}^{-1}\cdot M.$$ Therefore, for all $(\tau_1,\tau_2)\in\Gal_k\times\Gal_k$, the element $c_u(\tau_1,\tau_2):= u_{\tau_1} \tau_1(u_{\tau_2}) u_{\tau_1\tau_2}^{-1}$ (which depends on the choice of the representative $M$ and the family $u:=(u_\tau)_{\tau\in\Gal_k}$ satisfying, for all $\tau\in\Gal_k$, $u_\tau\cdot\tau(M) = M$) lies in the stabiliser of $M$ in $\G(\kb)$, which is $\Delta(\kb)$ since $M$ is assumed to be $\chi_\theta$-stable.

\begin{prop}\label{type_map}
The above map $$c_u: \Gal_k \times \Gal_k \lra \Delta(\kb),\quad
 (\tau_1,\tau_2) \lmt u_{\tau_1} \tau_1(u_{\tau_2}) u_{\tau_1\tau_2}^{-1}$$ is a normalised $\Delta(\kb)$-valued $2$-cocycle whose cohomology class only depends on the $\G(\kb)$-orbit of $M$, thus this defines a map $$\cT: \Modgs(\kb)^{\Gal_k} \lra H^2(\Gal_k;\Delta(\kb)) \simeq \Br(k)$$ that we shall call the type map.
\end{prop}

\begin{proof}
It is straightforward to check the cocycle relation $$c(\tau_1,\tau_2)c(\tau_1\tau_2,\tau_3) = \tau_1(c(\tau_2,\tau_3)) c(\tau_1,\tau_2\tau_3)$$ for all $\tau_1,\tau_2,\tau_3$ in $\Gal_k$. If we choose a different family $u':=(u'_\tau)_{\tau\in\Gal_k}$ such that $u'_\tau\cdot \tau(M) = M$ for all $\tau\in\Gal_k$, then $(u'_\tau)^{-1}\cdot M=u_\tau\cdot M$, thus $a_\tau:=u'_\tau u_\tau^{-1} \in \Delta(\kb)$ and it is straightforward to check, using that $\Delta(\kb)$ is a central subgroup of $\G(\kb)$, that $$u'_{\tau_1} \tau_1(u'_{\tau_2}) \big(u'_{\tau_1\tau_2}\big)^{-1} = \big(a_{\tau_1} \tau_1(a_{\tau_2}) a_{\tau_1\tau_2}^{-1}\big)\ \big(u_{\tau_1} \tau_1(u_{\tau_2}) u_{\tau_1\tau_2}^{-1}\big).$$ Therefore, the associated cocycles $c_u$ and $c_{u'}$ are cohomologous. If we now replace $M$ with $M'=g\cdot M$ for $g\in\G(\kb)$, then $$\tau(M') = \tau(g)\cdot\tau(M) = \tau(g) u_{\tau}^{-1} g^{-1}\cdot M'$$ and, if we set $u'_\tau:= g u_\tau \tau(g^{-1})$, we have $$c_{u'}(\tau_1,\tau_2) = g c_u(\tau_1,\tau_2) g^{-1} = c_u(\tau_1,\tau_2),$$ where the last equality follows again from the fact that $\Delta(\kb)$ is central in $\G(\kb)$. In particular, the two representatives $M$ and $M'$ give rise, for an appropriate choice of the families $u$ and $u'$, to the same cocycle, and thus they induce the same cohomology class $[c_u] = [c_{u'}]$.
\end{proof}

\noindent If $k$ is a finite field $\mathbb{F}_q$, then $\Br(\mathbb{F}_q)=0$. Other useful examples of target spaces for the type map are $\Br(\R)\simeq \Z/2\Z$ and $\Br(\Q_p)\simeq \Q/\Z$ for all prime $p$. Moreover, the group $\Br(\Q)$ fits in a canonical short exact sequence $$0\lra \Br(\Q) \lra \Br(\R) \oplus \bigoplus_{p\ \mathrm{prime}} \Br(\Q_p) \lra \Q/\Z \lra 0.$$

\begin{rmk}\label{type factors}
We note that the type map $\cT : \Modgs(k) \lra H^2(\Gal_k,\Delta(\kb))$ factors through the connecting homomorphism
\[ \delta : H^1(\Gal_k,\Gbar(\kb)) \lra H^2(\Gal_k,\Delta(\kb))\]
associated to the short exact sequence of groups $$1 \lra \Delta \lra \G \lra \Gbar:=\G/\Delta \lra 1.$$ By definition of $\cT$, for a $\Gal_k$-invariant orbit $\G(\kb) \cdot M$ in $\Modgs(\kb)^{\Gal_k}$, we choose elements $u_\tau \in \G(\kb)$ with $u_1 = 1_{\G}$ such that $ u_\tau \cdot \tau(M) = M$ for all $\tau \in \Gal_k$ and then construct a $\Delta(\kb)$-valued 2-cocycle $c_u(\tau_1,\tau_2) = u_{\tau_1} \tau_1(u_{\tau_2}) u_{\tau_1\tau_2}^{-1}$. If we let $\bar{u}_\tau$ denote the image of $u_\tau$ under the homomorphism $\G(\kb) \lra \Gbar(\kb)$, then $\bar{u} : \Gal_k \lra \Gbar(\kb)$ is a normalised 1-cocycle, as $u_{\tau_1} \tau_1(u_{\tau_2}) u_{\tau_1\tau_2}^{-1} \in \Delta(\kb)$ implies $\bar{u}_{\tau_1 \tau_2} =  \bar{u}_{\tau_1} \tau_1(\bar{u}_{\tau_2})$. Furthermore,
\[ [c_u] = \delta( [\bar{u}]). \]
As $[c_u]$ is independent of the choice of elements $u_\tau$ and representative $M$ of the orbit, and $\delta$ is injective, it follows that $[\bar{u}] \in H^1(\Gal_k,\Gbar(\kb))$ is also independent of these choices. Hence, the type map factors as $\cT = \delta \circ \cT'$ where
\[ \cT' :  \Modgs(k) \lra  H^1(\Gal_k,\Gbar(\kb)).\] This observation will be useful in $\S$\ref{rat_pts_that_do_not_come_from_rat_reps}. Note that, unlike that of $\cT$, the target space of $\cT'$ depends on $Q$.
\end{rmk}

\begin{rmk}[Intrinsic definition of the type map]
The presentation of $\Modgs(\kb)$ as the orbit space $\Reps(\kb)/\G(\kb)$ is particularly well-suited for defining the type map, as the stabiliser in $\G(\kb)$ of a point in $\Reps(\kb)$ is isomorphic to the automorphism group of the associated representation of $Q$. We can intrinsically define the type map, without using this orbit space presentation, as follows. A point in $\Modgs(\kb)$ corresponds to an isomorphism class of a $\theta$-geometrically stable $\kb$-representation $W$, and this point is fixed by $\Gal_k$-action if, for all $\tau\in\Gal_k$, there is an isomorphism $u_\tau: W\lra W^\tau$. The relation $W^{\tau_1\tau_2} = (W^{\tau_1})^{\tau_2}$ then implies that $\tilde{c}_u(\tau_1,\tau_2):=u_{\tau_1\tau_2}^{-1} u_{\tau_1}^{\tau_2} u_{\tau_2}$ is an automorphism of $W$. Once $\Aut(W)$ is identified with $\kb^\times$, this defines a $\kb^\times$-valued $2$-cocycle $\tilde{c}_u$, whose cohomology class is independent of the choice of the isomorphisms $(u_\tau)_{\tau\in\Gal_k}$ and the identification $\Aut(W)\simeq\kb^\times$.
\end{rmk}

\noindent We now use the type map to analyse which $k$-points of the moduli scheme $\Modgs$ actually correspond to $k$-representations of $Q$.

\begin{thm}\label{image_of_k_rep_in_k_pts_of_the_moduli_scheme}
The natural map $\Fgs(k) \lra \Modgs(k)$ induces a bijection $$\Fgs(k) \overset{\simeq}{\lra} \cT^{-1}([1]) \subset \Modgs(k)$$ from the set of isomorphism classes of $\theta$-geometrically stable $d$-dimensional $k$-representations of $Q$ onto the fibre of the type map $\cT:\Modgs(k)\lra \Br(k)$ over the trivial element of the Brauer group of $k$.
\end{thm}

\begin{proof}
Identify this map with $\fGalk$; then it is injective by Proposition \ref{fibres_of_fGalk}. If $\G(\kb)\cdot M$ lies in $\mathrm{Im}\, \fGalk$, we can choose a representative $M\in \Reps(\kb)^{\Gal_k}$, so the relation $u_\tau\cdot\tau(M) = M$ is trivially satisfied if we set $u_\tau=1_{\Gal_k}$ for all $\tau\in\Gal_k$. But then $c_u(\tau_1,\tau_2) \equiv 1_{\Delta(\kb)}$ so, by definition of the type map, $\cT(\G(\kb)\cdot M) = [c_u] = [1]$, which proves that $\mathrm{Im}\,\fGalk \subset \cT^{-1}([1])$. Conversely, take $M\in \Reps(\kb)$ with  $\G(\kb)\cdot M \in \cT^{-1}([1])$. By definition of the type map, this means that there exists a family $(u_\tau)_{\tau\in \Gal_k}$ of elements of $\G(\kb)$ such that $u_{1_{\Gal_k}} = 1_{\G(\kb)}$, $u_\tau\cdot \tau(M) = M$ for all $\tau\in\Gal_k$ and $c_u(\tau_1,\tau_2) := u_{\tau_1} \tau_1(u_{\tau_2}) u_{\tau_1\tau_2}^{-1}\in\Delta(\kb)$ for all $(\tau_1,\tau_2)\in\Gal_k\times\Gal_k$, and $[c_u]=[1]$, as $\cT(\G(\kb)\cdot M) = [c_u]$ by construction of $\cT$. By suitably modifying the family $(u_\tau)_{\tau\in\Gal_k}$ if necessary, we can thus assume that $u_{\tau_1}\tau_1(u_{\tau_2}) = u_{\tau_1\tau_2}$, which means that $(u_\tau)_{\tau\in\Gal_k}$ is a $\G(\kb)$-valued $1$-cocycle for $\Gal_k$. As $\G(\kb) = \prod_{v\in V} \GL_{d_v}(\kb)$, we have $$H^1(\Gal_k; \G(\kb)) \simeq \prod_{v\in V} H^1(\Gal_k; \GL_{d_v}(\kb))$$ so, by a well-known generalisation of Hilbert's 90th Theorem, $H^1(\Gal_k;\G(\kb))=1$ (for instance, see \cite[Proposition X.1.3 p.151]{Serre_local_fields}).  Therefore, there exists $g\in\G(\kb)$ such that $u_\tau= g\tau(g^{-1})$ for all $\tau\in\Gal_k$. In particular, the relation $u_\tau\cdot\tau(M) =M$ implies that $\tau(g^{-1}\cdot M) = g^{-1}\cdot M$, i.e.\ $(g^{-1}\cdot M) \in \Rep(\kb)^{\Gal_k}$, which shows that $\cT^{-1}([1]) \subset \mathrm{Im}\, \fGalk$.
\end{proof}

\begin{ex}
If $k$ is a finite field (so, in particular, $k$ is perfect and $\Br(k) = 1$), then $\Fgs(k) \simeq \Modgs(k)$: the set of isomorphism classes of $\theta$-geometrically stable $d$-dimensional $k$-representations of $Q$ is the set of $k$-points of a $k$-variety $\Modgs$.
\end{ex}

\subsection{Rational points that do not come from rational representations}\label{rat_pts_that_do_not_come_from_rat_reps}

If the Brauer group of $k$ is non-trivial, the type map $\cT:\Modgs(k) \lra \Br(k)$ can have non-empty fibres other than $\cT^{-1}([1])$; see Example \ref{quaternionic_rep}. In this case, by Theorem \ref{image_of_k_rep_in_k_pts_of_the_moduli_scheme}, the natural map $\Fgs(k)\lra \Modgs(k)$ is injective but not surjective. The goal of the present section is to show that the fibres of the type map over non-trivial elements of the Brauer group of $k$ admit a modular interpretation, using representations over division algebras

If $[c]\in H^2(\Gal_k;\kb^\times)$ lies in the image of the type map, then by definition there exists a representation $M\in\Reps(\kb)$ and a family $(u_\tau)_{\tau\in\Gal_k}$ such that $u_{1_{\Gal_k}}=1_{\G(\kb)}$ and $u_\tau\cdot \Phi_\tau(M)=M$ for all $\tau\in\Gal_k$. Moreover, the given $2$-cocycle $c$ is cohomologous to the $2$-cocycle $c_u:(\tau_1,\tau_2) \lmt u_{\tau_1} \Psi_{\tau_1}(u_{\tau_2}) u_{\tau_1\tau_2}^{-1}$. In order to analyse such families $(u_\tau)_{\tau\in\Gal_k}$ in detail, we introduce the following terminology, reflecting the fact that these families will later be used to modify the $\Gal_k$-action on $\Rep(\kb)$ and $\G(\kb)$.

\begin{defn}\label{modifying_fmly_def}
A \textit{modifying family} $(u_\tau)_{\tau\in\Gal_k}$ is a tuple, indexed by $\Gal_k$, of elements $u_\tau\in\G(\kb)$ satisfying:
\begin{enumerate}
\item $u_{1_{\Gal_k}} = 1_{\G(\kb)}$;
\item For all $(\tau_1,\tau_2)\in\Gal_k \times \Gal_k$, the element $c_u(\tau_1,\tau_2) := u_{\tau_1} \Psi_{\tau_1}(u_{\tau_2}) u_{\tau_1\tau_2}^{-1}$ lies in the subgroup $\Delta(\kb)\subset \G(\kb)$.
\end{enumerate}
\end{defn}

\noindent In particular, if $u=(u_\tau)_{\tau\in\Gal_k}$ is a modifying family, then the induced map $$c_u:\Gal_k\times\Gal_k \lra \Delta(\kb)$$ is a normalised $\Delta(\kb)$-valued $2$-cocycle. We now show that a modifying family can indeed be used to define new $\Gal_k$-actions on $\Rep(\kb)$ and $\G(\kb)$.

\begin{prop}\label{modified_actions_Gal_case}
Let $u=(u_\tau)_{\tau\in\Gal_k}$ be a modifying family in the sense of Definition \ref{modifying_fmly_def}. Then we can define modified $\Gal_k$-actions
$$
\Phi^u:  \Gal_k \times \Rep(\kb)  \lra  \Rep(\kb),\quad (\tau,M)  \lmt  u_\tau \cdot \Phi_\tau(M) $$ and 
$$
 \Psi^u:  \Gal_k \times \G(\kb) \lra \G(\kb),\quad (\tau,g)  \lmt  u_\tau \Psi_\tau(g) u_\tau^{-1}$$ which are compatible in the sense of \eqref{comp_btw_the_actions_Gal_case} and such that the induced $\Gal_k$-actions on $\Mod(\kb)$ and $\Modgs(\kb)$ coincide with the previous ones, constructed in \eqref{induced_Gal_k_action_on_geom_quotient}.
\end{prop}

\begin{proof}
The proof is a simple verification, using the fact that $\Delta(\kb)$ acts trivially on $\Rep(\kb)$ and is central in $\G(\kb)$, then proceeding as in Proposition \ref{Gal_preserves_GIT_sst} to show that the modified $\Gal_k$-action is compatible with semistability and stability of $\kb$-representations.
\end{proof}

\noindent Let us denote by $ _u\Rep(\kb)^{\Gal_k}$ the fixed-point set of $\Phi^u$ in $\Rep(\kb)$ and by $ _u\G(\kb)^{\Gal_k}$ the fixed subgroup of $\G(\kb)$ under $\Psi^u$. Proposition \ref{modified_actions_Gal_case} then immediately implies that $ _u\G(\kb)^{\Gal_k}$ acts on $ _u\Rep(\kb)^{\Gal_k}$ and that the map $f_{\Gal_k,u}$ taking the $ _u\G(\kb)^{\Gal_k}$-orbit of a $\theta$-geometrically stable representation $M\in _u\Reps(\kb)^{\Gal_k}$ to its $\G(\kb)$-orbit in $\Mods(\kb)$ lands in $\cT^{-1}([c_u])$, since one has $u_\tau\cdot \tau(M) = M$ for such a representation. We then have the following generalisation of Theorem \ref{image_of_k_rep_in_k_pts_of_the_moduli_scheme}.

\begin{thm}\label{fibres_of_the_type_map}
Let $(u_\tau)_{\tau\in\Gal_k}$ be a modifying family in the sense of Definition \ref{modifying_fmly_def} and let $$c_u:\Gal_k\times\Gal_k \lra \Delta(\kb)\simeq \kb^\times$$ be the associated $2$-cocycle. Then the map
$$f_{\Gal_k,u}:  _u\Reps(\kb)^{\Gal_k} /  _u\G(\kb)^{\Gal_k} \lra \cT^{-1}([c_u]),\quad
 _u\G(\kb)^{\Gal_k} \cdot M \lmt \G(\kb)\cdot M$$ is bijective.
\end{thm}

\begin{proof}
As $\Delta(\kb)$ is central in $\G(\kb)$, the action induced by $\Psi^u$ on $\Delta(\kb)$ coincides with the one induced by $\Psi$, so the injectivity of $f_{\Gal_k,u}$ can be proved as in Proposition \ref{fibres_of_fGalk}. The proof of surjectivity is then exactly the same as in Theorem \ref{image_of_k_rep_in_k_pts_of_the_moduli_scheme}. The only thing to check is that $H^1_u(\Gal_k;\G(\kb))=1$, where the subscript $u$ means that $\Gal_k$ now acts on $\G(\kb)$ via the action $\Psi^u$; this follows from the proof of \cite[Proposition X.1.3 p.151]{Serre_local_fields} once one observes that, if one sets $\Psi^u_\tau(x) := u_\tau 
\tau(x)$ for all $x\in \kb^{d_v}$, then one still has, for all $A\in\GL_{d_v}(\kb)$ and all $x\in \kb^{d_v}$, $\Psi^u_\tau(Ax) = \Psi^u_\tau(A) \Psi^u_\tau(x)$. After that, the proof is the same as in \textit{loc.\ cit.}.
\end{proof}

\noindent By Theorem \ref{fibres_of_the_type_map}, we can view the fibre $\cT^{-1}([c_u])$ as the set of isomorphism classes of $\theta$-geometrically stable, $(\Gal_k,u)$-invariant, $d$-dimensional $\kb$-representations of $Q$. Note that, in the context of $(\Gal_k,u)$-invariant $\kb$-representations of $Q$, semistability is defined with respect to $(\Gal_k,u)$-invariant $\kb$-subrepresentations only. However, analogously to Proposition \ref{sst_and_field_ext}, this is in fact equivalent to semistability with respect to all subrepresentations. The same holds for geometric stability, by definition. We have thus obtained a decomposition of the set of $k$-points of $\Modgs$ as a disjoint union of moduli spaces, completing the proof of Theorem \ref{decomp_thm_gal_intro}.


In order to give a more intrinsic modular description of each fibre of the type map $\cT^{-1}([c_u])$ appearing in the decomposition of $\Modgs(k)$ given by Theorem \ref{decomp_thm_gal_intro}, we recall that the Brauer group of $k$ is also the set of isomorphism classes of central division algebras over $k$, or equivalently the set of Brauer equivalence classes of central simple algebras over $k$. The dimension of any central simple algebra $A$ over $k$ is a square and the index of $A$ is then $ \ind\,(A) := \sqrt{\dim_k(A)}$.

\begin{prop}\label{nec_con_div_alg}
Assume that a central division algebra $D \in \Br(k)$ lies in the image of the type map 
\[ \cT: \Modgs(k) \lra H^2(\Gal_k; \ov{k}^\times) \cong \Br(k). \]
Then the index of $D$ divides the dimension vector; that is, $d = \ind\,(D)d'$ for some dimension vector $d' \in \NN^{V}$.
\end{prop}

\begin{proof}
We recall from Remark \ref{type factors} that $\cT$ has the following factorisation
\[ \xymatrix@1{\cT: \Modgs(k) \ar[r]^{\cT' \quad} & H^1(\Gal_k,\Gbar(\kb)) \ar[r]^{\delta} & H^2(\Gal_k, \Delta(\kb)) }, \]
where for a $\Gal_k$-invariant orbit $\G(\kb) \cdot M$ in $\Modgs(\kb)^{\Gal_k}$, we choose elements $u_\tau \in \G(\kb)$ for all $\tau \in \Gal_k$ such that $u_1 = 1_{\G}$ and $ u_\tau \cdot \tau(M) = M$, which determines a $\Gbar(\kb)$-valued 1-cocycle $\bar{u} : \Gal_k \lra \Gbar(\kb)$ such that 
\[\cT'(\G(\kb) \cdot M) = [\bar{u}].\]

For each vertex $v \in V$, the projection $\G \lra \GL_{d_v}$ maps $\Delta$ to the central diagonal torus $\Delta_v \subset \GL_{d_v}$, and so there is an induced map $\Gbar \lra \PGL_{d_v}$. In particular, this gives, for all $v\in V$, a commutative diagram
\begin{equation}\label{comp_btw_CSA_at_vertices} 
\xymatrix@1{H^1(\Gal_k,\Gbar(\kb)) \ar[r]^{}\ar[d] & H^1(\Gal_k,\PGL_{d_v}(\kb)) \ar[d]^{} \\ H^2(\Gal_k, \Delta(\kb)) \ar[r]^{\cong} & H^2(\Gal_k, \Delta_v(\kb)). } 
\end{equation} Since, by the Noether-Skolem Theorem, $\PGL_{d_v}(\kb)\simeq\Aut(\Mat_{d_v}(\kb))$, we can view $H^1(\Gal_k,\PGL_{d_v}(\kb))$ as the set of central simple algebras of index $d_v$ over $k$ (up to isomorphism): the class $[\bar{u}] \in H^1(\Gal_k,\Gbar(\kb))$ then determines, for each $v \in V$, an element $[\bar{u}_v] \in H^1(\Gal_k,\PGL_{d_v}(\kb))$, which in turn corresponds to a central simple algebra $A_v$ over $k$, of index $d_v$. Moreover, if $[\bar{u}]$ maps to the division algebra $D$ in $\Br(k)$, then we have, by the commutativity of Diagram \eqref{comp_btw_CSA_at_vertices}, that $D$ is Brauer equivalent to $A_v$ for all vertices $v$ (that is, $A_v \simeq M_{d'_v}(D)$ for some $d'_v\geq 1$). If $e :=\ind\, (D)$, then $\dim_k A_v=(\dim_k D)(\dim_D A_v) = e^2 {d'_v}^2$ so $\ind\, (A_v) = e d'_v$, i.e.\ $d_v=e d'_v$, for all $v \in V$. Thus, the index of $D$ divides the dimension vector $d$. 
\end{proof}

\noindent Consequently, we obtain the following sufficient condition for the decomposition of $\Modgs(k)$ to be indexed only by the trivial class in $\Br(k)$, in which case, all rational points come from rational representations.

\begin{cor}
Let $d \in \NN^V$ be a dimension vector which is not divisible by any of the indices of non-trivial central division algebras over $k$; then $\Modgs(k)$ is the set of isomorphism classes of $\theta$-geometrically stable $k$-representations of $Q$ of dimension $d$.
\end{cor}

\begin{ex}
For $k = \RR$, we have $\Br(\RR) = \{ \RR, \HH \}$ and  $\ind\,(\HH) = 2$; hence, for any dimension vector $d$ indivisible by $2$, the set $\Modgs(\RR)$ is the set of isomorphism classes of $\theta$-geometrically stable $\RR$-representations of $Q$ of dimension $d$.
\end{ex}

For a central division algebra $D \in \Br(k)$, we will interpret the fibre $\cT^{-1}(D)$ as the set of isomorphism classes of $\theta$-geometrically stable $D$-representations of $Q$ of dimension $d'$, where $d = \ind\,(D) d'$ (\textit{cf.}\ Theorem \ref{fibres div alg}). First we give some preliminary results about $D$-representations of $Q$ (by which we mean a representation of $Q$ in the category of $D$-modules). Note that, as $D$ is a skew field, the category $\cmod(D)$ of finitely generated $D$-modules behaves in the same way as a category of finite dimensional vector spaces over a field: $\cmod(D)$ is a semisimple Abelian category with one simple object $D$ and we can talk about the dimension of objects in $\cmod(D)$. We let $\crep_D(Q)$ denote the category of representations of $Q$ in the category $\cmod(D)$, and we let $\crep_D^d(Q)$ denote the subcategory of $d$-dimensional representations. Occasionally, we will encounter representations of $Q$ in the category of $A$-modules, where $A$ is a central simple algebra $A$ over $k$, but only fleetingly (see Remark \ref{rmk_on_Morita_equiv}).

Let $D \in \Br(k)$ be a division algebra. Recall that the connecting homomorphisms 
\begin{equation}\label{conn_homo_PGL}
H^1(\Gal_k,\PGL_e(\kb)) \overset{\delta_e}{\lra} H^2(\Gal_k;\kb^\times)
\end{equation}
associated for all $e \geq 1$ to the short exact sequences $$1\lra \kb^\times \lra \GL_e(\kb) \lra \PGL_e(\kb)\lra 1$$ induce a bijective map 
\begin{equation}\label{Brauer_gp_as_an_H1}
\varinjlim_e H^1(\Gal_k;\PGL_e(\kb)) \overset{\delta}{\lra} H^2(\Gal_k;\kb^\times)\simeq \Br(k)
\end{equation} (for example, see \cite[Corollary 2.4.10]{GS}), via which $D$ is given by a class $[\ov{a}_D]\in H^1(\Gal_k;\PGL_e(\kb))$ where $e:=\ind\,(D)$ is the index of $D$. We can then choose a $\GL_e(\kb)$-valued modifying family $a_D=(a_{D,\tau})_{\tau\in\Gal_k}$ such that, for each $\tau\in\Gal_k$, the element $\ov{a}_{D,\tau}\in\PGL_e(\kb)$ is the image of $a_{D,\tau}\in\GL_e(\kb)$ under the canonical projection. If we denote by $c_{a_D}$ the $\GG_m(\kb)$-valued $2$-cocycle associated to the modifying family $a_D$ (see Definition \ref{modifying_fmly_def}), we have $[c_{a_D}]=\delta([\ov{a}_D])=D$ in $\Br(k)$. In particular, the class $[c_{a_D}]$ is independent of the modifying family $[a_D]$ chosen as above.

\begin{rmk}\label{rmk_on_Morita_equiv}
Let $D \in \Br(k)$ be a central division algebra of index $e$. Since $\Br(\kb) = \{ 1\}$, the central simple algebra $\kb\otimes_k D$ over $\kb$ is Brauer equivalent to $\kb$; that is, $\kb\otimes_k D \cong \Mat_{e}(\kb)$. So, if $W$ is a $d'$-dimensional $D$-representation of $Q$, then we can think of $\kb\otimes_k W$ as a $d'$-dimensional $\Mat_{e }(\kb)$-representation of $Q$. For an algebra $R$, under the Morita equivalence of categories $\cmod(\Mat_{e}(R)) \simeq \cmod(R)$, the $\Mat_{e }(R)$-module $\Mat_{e}(R)$ corresponds to the $R$-module $R^{e}$. So, for $d = ed'$, there is an equivalence of categories
\[ \crep_{\kb\otimes_k D}^{d'} (Q) \cong \crep_{\kb}^d(Q). \]
In particular, we can view $\kb \otimes_k W$ as a $d$-dimensional $\kb$-representation of $Q$. This point of view will be useful in the proof of Proposition \ref{exists_k_var_for_D_reps}. More generally, if $L/k$ is an arbitrary field extension, the central simple $L$-algebra $L\otimes_k D$ is isomorphic to a matrix algebra $\Mat_{e}(D_L)$, where $D_L$ is a central division algebra over $L$ uniquely determined up to isomorphism. By the Morita equivalence $\cmod(\Mat_{e}(D_L)) \simeq \cmod(D_L)$, we can view the $d'$-dimensional $\Mat_{e}(D_L)$-representation $L\otimes_k W$ as a $d$-dimensional representation of $Q$ over the central division algebra $D_L\in\Br(L)$, which is the point of view we shall adopt in Definition \ref{geometric_stability_for_D_reps}.
\end{rmk}

For a division algebra $D \in \Br(k)$, consider the functor $\brep_{Q,d',D} : \text{c-Alg}_k \ra \Sets$ (resp.\ $\baut_{Q,d',D} : \text{c-Alg}_k \ra \Sets$) assigning to a commutative $k$-algebra $R$ the set
\begin{equation}\label{k_var_for_D_reps}
\brep_{Q,d',D}(R) = \bigoplus_{a \in A} \Hom_{\cmod(R \otimes_k D)}(R \otimes_k D^{d'_{t(a)}}, R \otimes_k D^{d'_{h(a)}})
\end{equation} (resp.\ $\baut_{Q,d',D}(R) = \prod_{v \in V} \Aut_{\cmod(R \otimes_k D)}(R \otimes_k D^{d'_v})$), where $d'$ is any dimension vector. Note that if $D=k$, these are the functor of points of the $k$-schemes $\rep_{Q,d'}$ and $\mathbf{G}_{Q,d'}$ introduced in Section \ref{GIT_constr_section}. We will now show that, for all $D\in \Br(k)$, these functors are representable by $k$-varieties, using Galois descent over the perfect field $k$. Let $e:=\ind \,(D)$ and choose a 1-cocycle $[\bar{a}_D] \in H^1(\Gal_k,\PGL_e(\kb))$ whose image under $\delta$, the bijective map from \eqref{Brauer_gp_as_an_H1}, is $D$. For each $\tau \in \Gal_k$, pick a lift $a_{D,\tau} \in \GL_e(\kb)$ of $\bar{a}_{D,\tau}$. Let $d:= ed'$ and  consider the modified $\Gal_k$-action on the $\kb$-schemes $\rep_{Q,d,\kb}:=\spec\kb\times_k \Rep$ (resp.\ $\mathbf{G}_{Q,d,\kb}:= \spec \kb \times_k \G$) given by the modifying family $u_D = (u_{D,\tau})_{\tau \in \Gal_k}$ defined by
\begin{equation}\label{direct_sum_modif_family}
\GL_{d_v}(\kb) \ni u_{D,\tau,v} := \left. \left(\begin{array}{ccc} a_{D,\tau} & & 0 \\ & \ddots & \\ 0 & &a_{D,\tau} \end{array} \right) \right\} \mathrm{(}d'_v\ \mathrm{times)}
\end{equation}
 (\textit{cf.}\ Proposition \ref{modified_actions_Gal_case}). This descent datum is effective, as $\rep_{Q,d,\kb}$ is affine so we obtain a smooth affine $k$-variety $\rep_{Q,d',D}$ (resp.\ $\mathbf{G}_{Q,d',D}$) such that $$\spec \kb \times_k \rep_{Q,d',D} \simeq \rep_{Q,d,\kb}$$ (resp.\ $\spec\kb \times_k \mathbf{G}_{Q,d',D} \simeq \mathbf{G}_{Q,d,\kb}$); for example, see \cite[Section 14.20]{GW}. 
 
For a commutative $k$-algebra $R$, we let $\overline{R}:= \kb \otimes_k R$ and we note that there is a natural $\Gal_k$-action on $\brep_{Q,d',D}(\overline{R})$. Moreover, the natural map
\begin{equation}\label{galois_descent_on_functor}
\brep_{Q,d',D}(R) \ra \brep_{Q,d',D}(\overline{R})^{\Gal_k}
\end{equation} 
is an isomorphism, by Galois descent for module homomorphisms, and similarly, this map is an isomorphism for $\baut_{Q,d',D}$.

\begin{prop}\label{exists_k_var_for_D_reps}
Let $D\in \Br(k)$ be a division algebra of index $e:=\ind\,(D)$. For a dimension vector $d'$, we let $d:= ed'$. Then the functors $\brep_{Q,d',D}$ and $\baut_{Q,d',D}$ introduced in \eqref{k_var_for_D_reps} are representable, respectively, by the $k$-varieties $\rep_{Q,d',D}$ and $\mathbf{G}_{Q,d',D}$ defined as above using descent theory and the modifying family \eqref{direct_sum_modif_family}. In particular, we have 
\[ \rep_{Q,d',D}(k) = \bigoplus_{a \in A} \Hom_{\cmod(D)}(D^{d'_{t(a)}}, D^{d'_{h(a)}}) \simeq {_{u_D}}\rep_{Q,d}(\kb)^{\Gal_k} \]
and $$\mathbf{G}_{Q,d',D}(k) = \prod_{v\in V} \GL_{d_v}(D) \simeq {_{u_D}}\G(\kb)^{\Gal_k},$$ so that $\rep_{Q,d',D}(k) / \mathbf{G}_{Q,d',D}(k)$ is in bijection with the set of isomorphism classes of $d'$-dimensional representations of $Q$ over the division algebra $D$. Moreover, there is an algebraic action of $\mathbf{G}_{Q,d',D}$ on $\rep_{Q,d',D}$ over $k$.
\end{prop}

\begin{proof} 
We will prove that $\brep_{Q,d',D}$ is representable by the $k$-variety $\rep_{Q,d',D}$ obtained by descent theory from $\rep_{Q,d,\kb}$ using the modified Galois action associated to the modifying family \eqref{direct_sum_modif_family}. The analogous statement for $\baut_{Q,d',D}$ is proved similarly and the rest of the proposition is then clear. To prove the statement for $\brep_{Q,d',D}$, we need to check for all $R \in \text{c-Alg}_k$ that $\brep_{Q,d',D}(R) \simeq \rep_{Q,d}(R)$ (and these isomorphisms are functorial in $R$). By Galois descent and \eqref{galois_descent_on_functor}, it suffices to show for $\overline{R}:= \kb \otimes_k R$, that $\brep_{Q,d',D}(\overline{R}) \simeq \rep_{Q,d,\kb}(\overline{R})$ and that the natural Galois action on $\brep_{Q,d',D}(\kb)$ coincides with the $u_D$-modified Galois action on $\Rep(\kb)$ defined as in Proposition \ref{modified_actions_Gal_case} using the modifying family $u_D$ introduced in \eqref{direct_sum_modif_family}. By definition of $\brep_{Q,d',D}$, one has 
$$\brep_{Q,d',D}(\overline{R}) = \bigoplus_{a \in A} \Hom_{\cmod(\overline{R} \otimes_k D)}(\overline{R} \otimes_k D^{d'_{t(a)}}, \overline{R} \otimes_k D^{d'_{h(a)}}).$$ 
As the division algebra $D$ is in particular a central simple algebra over $k$, the $\kb$-algebra $\kb\otimes_k D$ is also central and simple (over $\kb$). Since $\kb$ is algebraically closed, this implies that $\kb\otimes_k D\simeq \Mat_{ e}(\kb)$, where $e=\ind(D)$. Likewise 
\[ (\overline{R} \otimes_k D^{d'_v}) \simeq  (\kb\otimes_k R) \otimes_k D^{d'_v} \simeq R \otimes_k (\kb \otimes_k D)^{d'_{v}} \simeq  R \otimes_k \Mat_{e}(\kb)^{d'_{v}} \simeq \Mat_{e}(\overline{R})^{d'_{v}}.\] 
Through these isomorphisms, the canonical Galois action $z\otimes x \lmt \tau(z) \otimes x$ on $\overline{R}\otimes_k D$ translates to $M\lmt a_{D,\tau} \tau(M) a_{D,\tau^{-1}}$ (see for instance \cite[Chapter 10, $\S$5]{Serre_local_fields}), where $M\in \Mat_{ e}(\overline{R})$ and the element $a_{D,\tau}\in \GL_e(\kb)$ belongs to a family that maps to $D\in \Br(k)$ under the isomorphism \eqref{Brauer_gp_as_an_H1} and is the same as the one used to define the modifying family $u_D$ in \eqref{direct_sum_modif_family}. Under the Morita equivalence of categories $\cmod(\Mat_{e}(\overline{R})) \simeq \cmod(\overline{R})$ recalled in Remark \ref{rmk_on_Morita_equiv}, the $\Mat_{e}(\overline{R})$-module $\Mat_{e}(\overline{R})^{d'_v}$ corresponds to $\overline{R}^{d_v}$, so we have $$\brep_{Q,d',D}(\overline{R}) \simeq \bigoplus_{a\in A} \Hom_{\cmod(\overline{R})}(\overline{R}^{d_{t(a)}}, \overline{R}^{d_{h(a)}}) = \rep_{Q,d,\kb}(\overline{R}).$$ The $\overline{R}$-module $\overline{R}^{d_v}\simeq (\overline{R}^e)^{d'_v}$ does not inherit a Galois action but instead a so-called $D$-structure (see Example \ref{quaternionic_rep} for the concrete, non-trivial example where $k=\R$ and $D=\mathbb{H}$) given, for all $\tau\in\Gal_k$, by $$ \Phi_{D,\tau}:  (\overline{R}^e)^{d'_v} \lra (\overline{R}^e)^{d'_v},\quad \big(x_1,\,\ldots\, ,x_{d'_v}\big)  \lmt  \big(a_{D,\tau}\tau(x_1),\,\ldots\, ,a_{D,\tau}\tau(x_{d'_v})\big)$$ where, for all $i\in\{1,\,\ldots\, ,d'_v\}$, we have $x_i\in \overline{R}^e$ and $a_{D,\tau}\in\GL_e(\kb)$, while $\tau\in \Gal_k$ acts component by component. This in turn induces a genuine Galois action on $\Hom_{\cmod(\overline{R})}(\overline{R}^{d_{t(a)}}, \overline{R}^{d_{h(a)}})$, given by $M_a\lmt u_{D,\tau,h(a)} \tau(M_a) u_{D,\tau,t(a)}^{-1}$, where $u_D$ is the $\G(\kb)$-valued modifying family defined in \eqref{direct_sum_modif_family}. In particular, this $\Gal_k$-action on $\Rep(\overline{R})$ coincides with the $\Gal_k$-action $\Phi^{u_D}$ of Proposition \ref{modified_actions_Gal_case}, which concludes the proof.
\end{proof}

We also note that if $D$ lies in the image of $\cT$, then there is a $\Gbar(\kb)$-valued 1-cocycle $\bar{u}$ mapping to $D$ under the connecting homomorphism by Remark \ref{type factors}. In this case, a lift $u=(u_{\tau} \in \G(\kb))_{\tau \in \Gal_k}$ of $\bar{u}$ is a modifying family, which we can use in place of the family $u_D$ given by \eqref{direct_sum_modif_family}, as $[\bar{u}] = [\bar{u}_D] \in H^1(\Gal_k,\Gbar(\kb))$.

\begin{rmk}\label{rmk nonperf field}
For an arbitrary field $k$ and a division algebra $D \in \Br(k)$, one can also construct a $k$-variety $\rep_{Q,d',D}$ (resp. $\mathbf{G}_{Q,d',D}$) representing the functor $\brep_{Q,d',D}$ (resp. $\baut_{Q,d',D}$) by Galois descent for $\Gal(k^s/k)$, where $k^s$ denotes a separable closure of $k$. More precisely, for $d := \ind\,(D)d'$, we use Galois descent for the modified $\Gal(k^s/k)$-action on $\rep_{Q,d,k^s}:= \spec k^s \times_k \Rep$ (resp. $\mathbf{G}_{Q,d,k^s} :=\spec k^s \times_k \G$) given by the family $u_D$ defined in \eqref{direct_sum_modif_family}. Then the above proof can be adapted, once we note that we can still apply Remark \ref{rmk_on_Morita_equiv}, as $\Br(k^s) = 0$. 
\end{rmk}

We now turn to notions of semistability for $D$-representations of $Q$. The slope-type notions of $\theta$-(semi)stability for $k$-representations naturally generalise to $D$-representations (or, in fact, representations of $Q$ in a category of modules), so we do not repeat them here. As the definition of geometric stability over a division algebra is not obvious, we write it out explicitly.

\begin{defn}\label{geometric_stability_for_D_reps} For a central division algebra $D$ of index $e$ over $k$, a $D$-representation $W$ of $Q$ is called $\theta$\textit{-geometrically stable} if, for all field extensions $L/k$, the representation $L\otimes_k W$ is $\theta$-stable as a $D_L$-representation, where $D_L\in\Br(L)$ is the unique central division algebra over $L$ such that $L\otimes_k D\simeq \Mat_{e}(D_L)$.
\end{defn}

We recall that a $k$-representation $W$ is $\theta$-geometrically stable if and only if $\kb \otimes_k W$ is $\theta$-stable. We can now prove that an analogous statement holds for representations over a division algebra $D \in \Br(k)$.

\begin{lemma}\label{lem_geom_s_div_alg}
Let $D$ be a division algebra over a perfect field $k$. Let $d'$ be a dimension vector and set $d:=  \ind\,(D)d'$. Let $W$ be a $d'$-dimensional $D$-representation of $Q$. By Remark \ref{rmk_on_Morita_equiv}, the representation $\kb\otimes_k W$ can be viewed as a $d$-dimensional representation of $Q$ over $\kb$. Then the following statements are equivalent:
\begin{enumerate}
\item $W$ is $\theta$-geometrically stable as a $d'$-dimensional $D$-representation of $Q$. 
\item $\kb \otimes_k W$ is $\theta$-stable as a $d$-dimensional $\kb$-representation of $Q$.
\end{enumerate}
\end{lemma}

\begin{proof} 
By definition of geometric stability, it suffices to show that if $\kb\otimes_k W$ is stable as a $\kb$-representation, then $W$ is geometrically stable. So let $L/k$ be a field extension. As in Proposition \ref{sst_and_field_ext}, it suffices to treat separately the case where $L/k$ is algebraic and the case it is purely transcendental of transcendence degree one. If $L/k$ is algebraic, we can assume that $L\subset \kb$ and we have that $\kb\otimes_L (L\otimes_k W) \simeq (\kb\otimes_k W)$, which is stable, so $L\otimes_k W$ is stable, as in Part (1) of Proposition \ref{sst_and_field_ext}. If $L\simeq k(X)$, let us show that $L\otimes_k W$ is stable as a $D_L$-representation. Since $\kb(X) \otimes_{k(X)} (k(X)\otimes_k W) \simeq \kb(X)\otimes_k W$, by the same argument as earlier it suffices to show that $\kb(X)\otimes_k W$ is stable as a $D_{\kb(X)}$-representation. We have that $\kb(X)\otimes_k W \simeq \kb(X)\otimes_{\kb}(\kb\otimes_k W)$. But since $\kb\otimes_k W$ is stable as a $\kb$-representation by assumption and $\kb$ is algebraically closed, $\kb\otimes_k W$ is geometrically stable by Corollary \ref{charac_of_geom_stability}, so $\kb(X)\otimes_{\kb}(\kb\otimes_k W)$ is stable and the proof is complete.
\end{proof}

We can now give a modular interpretation of our decomposition in Theorem \ref{decomp_thm_gal_intro}.

\begin{thm}\label{fibres div alg}
Let $k$ be a perfect field and $D \in \Br(k)$ be a division algebra in the image of the type map 
$\cT: \Modgs(k) \lra H^2(\Gal_k; \ov{k}^\times) \cong \Br(k)$; thus we have $d = \ind\,(D) d'$ and a modifying family $u_D$ such that $[c_{u_D}] = D \in \Br(k)$. Then
\[ \cT^{-1}(D) \cong {_{u_D}}\Reps(\kb)^{\Gal_k} / {_{u_D}}\G(\kb)^{\Gal_k} \cong \rep_{Q,d',D}^{\theta-gs}(k) / \mathbf{G}_{Q,d',D}(k), \]
where the latter is the set of isomorphism classes of $\theta$-geometrically stable $D$-representations of $Q$ of dimension $d'$.
\end{thm}

\begin{proof} 
The first bijection follows from Theorem \ref{fibres_of_the_type_map} and the second one follows from Proposition \ref{exists_k_var_for_D_reps} and Lemma \ref{lem_geom_s_div_alg}.
\end{proof}

\begin{rmk}
For a non-perfect field $k$ with separable closure $k^s$, one should not expect Theorem \ref{fibres div alg} to hold in its current form, because the $k^s$-points of $\Modgs$ do not necessarily correspond to isomorphism classes of $\theta$-geometrically stable $d$-dimensional $k^s$-representations (whereas they do for $k$ perfect, as the geometric points of $\Modgs$ are as expected). This problem is an artefact of $\Modgs$ being constructed as a GIT quotient. 
\end{rmk}

\begin{lemma}\label{stability_over_separably_closed_field_implies_geometric_stability}
Let $k$ be a separably closed field and $W$ be a $\theta$-stable $k$-representation of $Q$; then
\begin{enumerate}
\item $W$ is a simple $k$-representation, and thus $\Aut(W) \cong \GG_m$,
\item $W$ is $\theta$-geometrically stable.
\end{enumerate}
In particular, over a separably closed field, geometric stability and stability coincide.
\end{lemma}
\begin{proof}
As $W$ is $\theta$-stable, it follows that every endomorphism of $W$ is either zero or an isomorphism; thus $\End(W)$ is a division algebra over $k$. As $k$ is separably closed, $\Br(k) = 0$ and so $\End(W)=k$ and $\Aut(W) = \GG_m$. However, for a simple representation, stability and geometric stability coincide (for example, one can prove this by adapting the argument for sheaves in \cite[Lemma 1.5.10]{Huybrechts_Lehn} to quiver representations).
\end{proof}

\noindent Theorem \ref{thm_Galois_div_alg_intro} then follows immediately from Theorems \ref{decomp_thm_gal_intro} and \ref{fibres div alg}. Finally, let us explicitly explain this modular decomposition for the example of $k = \RR$.

\begin{ex}\label{quaternionic_rep}
Let $k=\R$ and let $[c] = -1\in \Br(\R)\simeq\{ 1,-1\} \simeq \{ \RR, \HH \}$. Then a modifying family corresponds to an element $u\in\G(\C) = \prod_{v\in V} \GL_{d_v}(\C)$ such that, for all $v\in V$, $u_v \overline{u_v} = -I_{d_v}$, implying that $|\det u_v|^2 = (-1)^{d_v}$, which can only happen if $d_v = 2 d'_v$ is even for all $v\in V$. We then have a quaternionic structure on each $\C^{d_v} \cong \HH^{d'_v}$, given by $x\lmt u_v \overline{x}$ and a modified $\Gal_{\R}$-action on $\Rep(\C)$, given by $(M_a)_{a\in A} \lmt u_{h(a)} \overline{M_a} u_{t(a)}^{-1}$. The fixed points of this involution are those $(M_a)_{a\in A}$ satisfying $u_{h(a)} \overline{M_a} u_{t(a)}^{-1}=M_a$, i.e.\ those $\C$-linear maps $M_a: W_{t(a)} \lra W_{h(a)}$ that commute with the quaternionic structures defined above, and thus are $\HH$-linear. The subgroup of $\G(\C) = \prod_{v\in V} \GL_{d_v}(\C)$ consisting, for each $v\in V$, of automorphisms of the quaternionic structure of $\C^{d_v}$ is the real Lie group $\G(\C)^{(\Gal_{\R},u)} = \prod_{v\in V} \mathbf{U}^*(d_v)$, where $\mathbf{U}^*(2n) = \GL_n(\HH)$. Hence, the fibre $\cT^{-1}(-1)$ of the type map is in bijection with the set of isomorphism classes of $\theta$-geometrically stable quaternionic representations of $Q$ of dimension $d'$.
\end{ex}

\section{Gerbes and twisted quiver representations}\label{sec_gerbes_and_twisted_reps}

\subsection{An interpretation of the type map via gerbes}\label{sec gerbes}

In this section, we give an alternative description of the type map using $\GG_m$-gerbes that works over any field $k$. The following result collects the relevant results that we will need on gerbes and torsors; for further details, see \cite[Chapter 12]{olsson}.

\begin{prop}\label{prop gerbes}
Let $\fX$ be an Artin stack over $k$ and let $G, G'$ and $G''$ be affine algebraic group schemes over $k$; then the following statements hold.
\begin{enumerate}
\item \cite[Corollary 12.1.5]{olsson} There is a natural bijection
\[ H^1_{\et}(\fX,G) \simeq \{ \text{isomorphisms classes of } G \text{-torsors over } \fX \}. \]
\item \cite[12.2.8]{olsson} For $G$ commutative, there is a natural bijection
\[ H^2_{\et}(\fX,G) \simeq \{ \text{isomorphisms classes of } G \text{-gerbes over } \fX \}. \]
\item \cite[Lemma 12.3.9]{olsson} A short exact sequence $1 \lra G' \lra G \lra G'' \lra 1$ with $G'$ commutative and central in $G$ induces an exact sequence
\[\xymatrix@1{ H^1_{\et}(\fX,G')  \ar[r] & H^1_{\et}(\fX,G)  \ar[r] & H^1_{\et}(\fX,G'')  \ar[r]^{\delta} & H^2_{\et}(\fX,G') .& }\] 
Moreover, the isomorphism class of a $G''$-torsor $\cP \lra \fX$, such that $\cP$ is representable by a $k$-scheme, has image under $\delta$ given by the class of the $G'$-gerbe $\cG_{G}(\cP)$ of liftings of $\cP$ to $G$ (\textit{cf}.\ Definition \ref{gerbe of liftings}).
\end{enumerate}
\end{prop}

\begin{defn}\label{gerbe of liftings}
For a short exact sequence $1 \lra G' \lra G \lra G'' \lra 1$ of affine algebraic groups schemes over $k$ with $G'$ abelian and a principal $G''$-bundle $\cP$ over an Artin stack $\fX$ over $k$, the gerbe $\cG_{G}(\cP)$ of liftings of $\cP$ to $G$ is the $G'$-gerbe over $\fX$ whose groupoid over $S \lra \fX$, for a $k$-scheme $S$, has objects given by pairs $(\cQ, f: \cQ \lra \cP_S)$ consisting of a principal $G$-bundle $\cQ$ over $S$ and an $S$-morphism $f : \cQ \lra \cP_S:= \cP \times_{\fX} S$ which is equivariant with respect to the homomorphism $G \lra G''$. An isomorphism between two objects $(\cQ,f)$ and $(\cQ',f')$ over $S$ is an isomorphism $\varphi : \cQ \lra \cQ'$ of $G$-bundles such that $f = f' \circ \varphi$.
\end{defn}

Let us now turn our attention to quiver representations and consider the stack of $d$-dimensional representations of $Q$ over an arbitrary field $k$, which is the quotient stack 
\[ \fM_{Q,d}=[\Rep/\G].\] 
Since $\Gbar = \G/\Delta $ and the group $\Delta \cong \GG_m$ acts trivially on $\Rep$, the natural morphism
\[ \pi: \fM_{Q,d}=[\Rep/\G] \lra \fX:= [\Rep/\Gbar] \]
is a $\GG_m$-gerbe. If we restrict this gerbe to the $\theta$-geometrically stable locus, then the base is a scheme rather than a stack, namely the moduli space of $\theta$-geometrically stable representations of $Q$
\[\pi^{\theta-gs}: \fM_{Q,d}^{\theta-gs} := [\Rep^{\theta-gs}/\G] \lra \Modgs = [\Rep^{\theta-gs}/\Gbar]. \]
The Brauer group $\Br(k)$ can also be viewed as the set of isomorphism classes of $\GG_m$-gerbes over $\spec k$, by using Proposition \ref{prop gerbes} and the isomorphism $H^2(\Gal_k, \kb^\times) \cong H^2_{\et}(\spec k,\GG_m)$ given by Grothendieck's Galois theory. 
By pulling back the $\GG_m$-gerbe $\pi$ along a point $r : \spec k \lra \fX$, we obtain a $\GG_m$-gerbe $\cG_r := r^*\fM_{Q,d} \lra \spec k$. This defines a morphism
\begin{equation}\label{def_G}
\cG : \fX(k) \lra H^2_{\et}(\spec k, \GG_m)\cong \Br(k), 
\end{equation}
whose restriction to the $\theta$-geometrically stable locus, we denote by
\[ \cG^{\theta-gs} :\Modgs(k) \lra H^2_{\et}(\spec k, \GG_m)\cong \Br(k).\]

In Corollary \ref{two_type_maps_agree}, we will show that $\cG^{\theta-gs}$ coincides with the type map 
\[\cT: \Modgs(k) \lra H^2(\Gal_k, \Delta(\kb))\cong \Br(k)\]
constructed above. In order to compare the above morphism $\cG^{\theta-gs}$ with the type map $\cT$, we recall from Remark \ref{type factors} that the type map factors as $\cT = \delta \circ \cT'$; that is,
\[ \xymatrix@1{\cT: \Modgs(k) \ar[r]^{\cT' \quad} & H^1(\Gal_k,\Gbar(\kb)) \ar[r]^{\delta} & H^2(\Gal_k, \Delta(\kb)) } \]
for the connecting homomorphism $\delta$ associated to the short exact sequence $$1 \lra \Delta \lra \G \lra \Gbar \lra 1.$$ We will also refer to $\cT'$ as the type map. 

Let us describe a similar factorisation of $\cG$. The morphism $p: \Rep \lra \fX=[\Rep/\Gbar]$ is a principal $\Gbar$-bundle 
and determines a map
\begin{equation}\label{def_P} 
\cP:   \fX(k)  \lra  H^1_{\et}( \spec k, \Gbar) ,\quad  r  \longmapsto  [\cP_r], 
\end{equation}
where $\cP_r$ is the $\Gbar$-bundle $\cP_r:=r^*\Rep \lra \spec k$. We denote the restriction of $\cP$ to the $\theta$-geometrically stable subset by
\[  \cP^{\theta-gs}:   \Modgs(k) \lra  H^1_{\et}( \spec k, \Gbar). \]

By a slight abuse of notation, we will use $\delta$ to denote both the connecting maps 
\[ \delta: H^1_{\text{\'{e}t}}( \spec k, \Gbar) \lra H^2_{\text{\'{e}t}}( \spec k, \GG_m) \]
and
\[ \delta: H^1_{\et}( \fX, \Gbar) \lra H^2_{\et}( \fX, \GG_m)  \]
in \'{e}tale cohomology given by the exact sequence $1 \lra \Delta \lra \G \lra \Gbar \lra 1$. 

\begin{lemma}\label{delta universally}
The $\GG_m$-gerbe $\cG_{\G}(\Rep) \lra \fX$ of liftings of the principal $\Gbar$-bundle $p:  \Rep \lra \fX$ to $\G$ is equal to $\fM_{Q,d} \lra \fX$. In particular, we have
\[ \delta([\Rep]) = [\fM_{Q,d}]. \]
\end{lemma}

\begin{proof}
Let us write $ \mathscr{P}:= \Rep \lra \fX$ and $\cG := \cG_{\G}( \mathscr{P})$; then we will construct  isomorphisms
\[ \alpha: \cG  \rightleftarrows \fM_{Q,d} : \beta \]
of stacks over $\fX$. First, we recall that $ \fM_{Q,d} = [\mathscr{P} /\G]$ is a quotient stack, and so, for a $k$-scheme $S$, its $S$-valued points are pairs $(\cQ, h : \cQ \lra  \mathscr{P})$ consisting of a principal $\G$-bundle $\cQ$ over $S$ and a $\G$-equivariant morphism $h$.

Let $S \lra \fX$ be a morphism from a scheme $S$; then we define the functor $\alpha_S :  \cG(S) \lra \fM_{Q,d}(S)$ as follows. For an object $(\cQ,f: \cQ \lra  \mathscr{P}_S) \in  \cG(S)$, we can construct a morphism $h: \cQ \lra  \mathscr{P}$ as the composition of $f$ with the projection $ \mathscr{P}_S \lra  \mathscr{P}$. As $f$ is equivariant with respect to $\G \lra \Gbar$ and $\Delta$ acts trivially on $ \mathscr{P}$, it follows that $h$ is $\G$-equivariant. Thus $\alpha_S(\cQ,f):=(\cQ,h) \in  \fM_{Q,d}(S)$. Since the isomorphisms on both sides are given by isomorphisms of $\G$-bundles over $S$ satisfying the appropriate commutativity properties, it is clear how to define $\alpha_S$ on isomorphisms. Conversely, to define $\beta_S$, we take an object $(\cQ, h : \cQ \lra  \mathscr{P}) \in \fM_{Q,d}(S)$ given by a $\G$-bundle $\cQ$ over $S$ and a $\G$-equivariant map $h$. By the universal property of the fibre product $ \mathscr{P}_S= \mathscr{P} \times_{\fX} S$, a morphism $h: \cQ \lra  \mathscr{P}$ is equivalent to a $S$-morphism $f: \cQ \lra  \mathscr{P}_S$, where here we use the fact that $ \mathscr{P}, S$ and $\mathscr{P}_S$ are all $k$-schemes, so that this $S$-morphism is unique. Since $\Delta$ acts trivially $\mathscr{P}= \Rep$, the $\G$-equivariance of $h$ is equivalent to $h$ being equivariant with respect to the homomorphism $\G \lra \Gbar$; thus $f$ is also equivariant for this homomorphism. Hence $\beta_S(\cQ,h) := (\cQ,f) \in \cG(S)$. From their constructions, it is clear that $\alpha$ and $\beta$ are inverses.

The final statement follows from Proposition \ref{prop gerbes}.
\end{proof}

\begin{cor}\label{G factors}
The following triangle commutes
\[ \xymatrixcolsep{0.2pc}\xymatrix@1{  [\Rep/\Gbar](k) \ar[rr]^{\cG} \ar[rd]_{\cP} & & H^2_{\et}( \spec k, \GG_m). \\ 
& H^1_{\et}( \spec k, \Gbar ).  \ar[ru]_{\delta} }\] 
\end{cor}
\begin{proof}
Since $\cG$ (resp.\ $\cP$) is defined by pointwise pulling back the $\GG_m$-gerbe $\pi : \fM_{Q,d} \lra \fX$ (resp.\ the $\Gbar$-bundle $p: \Rep \lra \fX$), this follows immediately from Lemma \ref{delta universally}.
\end{proof}

Consequently, it will suffice to compare the maps $\cP^{\theta-gs}$ and $\cT'$. Let us explicitly describe the \v{C}ech cocycle representing $[\cP_r]$ for $r \in \fX(k)$. We pick a finite separable extension $L/k$ such that $(\cP_r)_L \lra \spec L$ is a trivial $\Gbar$-bundle; that is, it admits a section $\sigma \in  \cP_r(L) \subset \Rep(L)$, which corresponds to a  $L$-representation $W_{r}$ of $Q$. Over $\spec(L \otimes_k L)$, the transition functions determine a cocycle $\varphi \in \Gbar(L \otimes_k L)$ such that $p_1^*\sigma =  \varphi \cdot p_2^*\sigma$ in $\cP_r(L \otimes_k L)$.  Then $\varphi$ is a \v{C}ech cocycle whose cohomology class in $H^1_{\et}(\spec k, \Gbar)$ represents the $\Gbar$-torsor $\cP_r$.

For $k$ perfect, let us recall the relationship between \'{e}tale cohomology and Galois cohomology given by Grothendieck's generalised Galois theory (\textit{cf.}\ \cite[Tag 03QQ]{stacks-project}). For all finite Galois extensions $L/k$, the isomorphisms
\[  h : \Gal_{L/k} \times \spec L  \lra  \spec L \times_{\spec k} \spec L,\quad (\tau,s)  \longmapsto h_{\tau}(s):=(s,\tau^*(s)) \]
induce isomorphisms $\gamma :  H^i( \Gal_k, G(\kb)) \cong H^i_{\et}(\spec k, G)$ for $i =1$ and any affine group scheme $G$ over $k$, and for $i=2$ and $G/k$ a commutative group scheme.

\begin{prop}\label{compare P and T dash}
Let $k$ be a perfect field; then the type map $\cT' :  \Modgs(k)  \lra  H^1( \Gal_k, \Gbar(\kb))$ agrees with the map $\cP^{\theta-gs}:  \Modgs(k)  \lra  H^1_{\et}(\spec k, \Gbar)$ under the isomorphism $H^1_{\et}(\spec k, \Gbar) \cong H^1( \Gal_k, \Gbar(\kb))$. 
\end{prop}
\begin{proof}
Let $r \in  \Modgs(k)$; then the $\Gbar$-bundle $\cP_r:=r^*\Rep \lra \spec k$ trivialises over some finite separable extension $L/k$ as above, and we can assume that $L/k$ is a finite Galois extension, by embedding $L/k$ in a Galois extension if necessary. Then there is a section $\sigma \in  \cP_r(L) \subset \Rep(L)$ corresponding to a  $L$-representation $W$ of $Q$, and the transition maps are encoded by a cocycle  $\varphi \in \Gbar(L \otimes_k L)$ such that $p_1^*\sigma =  \varphi \cdot p_2^*\sigma$. Under the isomorphism $\gamma$, the cocycle $\varphi$ is sent to a $1$-cocycle $u_L: \Gal_{L/k} \lra \Gbar(L)$ such that, for $\tau_L \in \Gal_{L/k}$, we have
\[ h_{\tau_L}^* \varphi = u_{L,\tau_L} \in \Gbar(L), \]
for the morphism $h_{\tau_L} : \spec L \lra \spec L \times_k \spec L$ described above. Furthermore, by pulling back the equality $p_1^*\sigma =\varphi \cdot p_2^*\sigma $ along $h_{\tau_L}$, for each $\tau_L \in \Gal_{L/k}$, we obtain an equality
\[ W = u_{L,\tau_L} \cdot \tau_L(W) \]
for all $\tau_L \in \Gal(L/k)$. Hence, the orbit $\Gbar(L) \cdot W$ is $\Gal_{L/k}$-fixed. 

By precomposing the $1$-cocycle $u_L: \Gal_{L/k} \lra \Gbar(L)$ with the homomorphism $\Gal_k \lra \Gal_{L/k}$ and postcomposing with the inclusion $\Gbar(L) \hookrightarrow \Gbar(\kb)$, we obtain a new $1$-cocycle
\[ u: \Gal_{k} \lra \Gbar(\kb). \]
For $\tau \in  \Gal_{k}$, we let $\tau_L$ denote the image of $\tau$ under $\Gal_k \lra \Gal_{L/k}$. Then 
\[  u_\tau \cdot \tau (W \otimes_L \kb) = u_\tau \cdot (\tau_L(W) \otimes_L \kb) = (u_{L,\tau_L} \cdot \tau_L(W)) \otimes_L \kb = W \otimes_L \kb. \]
Thus $\Gbar(\kb) \cdot (W \otimes_L \kb) \in \Modgs(\kb)^{\Gal_k}$ and this Galois fixed orbit corresponds to the $k$-rational point $r \in \Modgs(k)$. Moreover, by construction of $\cT'$, we have $\cT'(r) = [u]$ (\textit{cf.}\ Remark \ref{type factors}).
\end{proof}

\begin{cor}\label{two_type_maps_agree}
Under the isomorphism $H^2_{\et}(\spec k, \GG_m) \cong H^2( \Gal_k, \GG_m(\kb))$, the type map for a perfect field $k$
\[ \cT :\Modgs(k) \lra H^2(\Gal_k, \GG_m(\kb))\cong \Br(k) \]
coincides with the map 
\[ \cG^{\theta-gs} :\Modgs(k) \lra H^2(\spec k, \GG_m)\cong \Br(k)\]
determined by the $\GG_m$-gerbe $\pi^{\theta-gs}: \fM_{Q,d}^{\theta-gs}\lra \Modgs$.
\end{cor}

\begin{proof}
This follows from Proposition \ref{compare P and T dash}, Remark \ref{type factors} and Corollary \ref{G factors}. 
\end{proof}

Both points of view are helpful: the definition of the type map $\cT$ using the GIT construction of $\Modgs$ is useful due to its explicit nature, whereas the definition of the map $\cG$ using the $\GG_m$-gerbe $\fM_{Q,d}^{\theta-gs} \lra \Modgs$ is more conceptual.

\subsection{Twisted quiver representations}

In this section, we define a notion of twisted quiver representations over an arbitrary field $k$ (where the twisting is given by an element in the Brauer group $\Br(k)$) analogous to the notion of twisted sheaves due to C\u{a}ld\u{a}raru, de Jong and Lieblich \cite{cald,dJ,Lieblich}. 


Let $\alpha : \fZ \lra  \spec k$ be a $\GG_m$-gerbe. For an \'{e}tale cover $S=\spec L \lra \spec k$ given by a finite separable extension $L/k$, we let $S^2 := S \times_k S = \spec (L \otimes_k L)$ and $S^3 := S \times_k S \times_k S$ and so on. We use the notation $p_1,p_2 : S^2 \lra S$ and $p_{ij} : S^3 \lra S^2$ to denote the natural projection maps. We will often use such an \'{e}tale cover to represent $\alpha$ by a \v{C}ech cocycle $\alpha \in \Gamma(S^3, \GG_m) = (L \otimes_k L \otimes_k L)^\times$ whose pullbacks to $S^4$ satisfy the natural compatibility conditions.

Let us first give a definition of twisted quiver representations, which is based on C\u{a}ld\u{a}raru's definition of twisted sheaves. The definition based on Lieblich's notion of twisted sheaves is discussed in Remark \ref{Lieblich twisted}.

\begin{defn}\label{def_twisted_rep}
Let $\alpha : \fZ \lra  \spec k$ be a $\GG_m$-gerbe and take an \'{e}tale cover $S=\spec L \lra \spec k$, such that $\alpha$ is represented by a \v{C}ech cocycle $\alpha \in \Gamma(S^3, \GG_m)$. Then an $\alpha$-twisted $k$-representation of $Q$ (with respect to this presentation of $\alpha$ as a \v{C}ech cocycle) is a tuple $(W,\varphi)$ consisting of an $L$-representation $W$ of $Q$ and an isomorphism $\varphi : p_1^*W \lra p_2^*W$ of $L \otimes_k L$-representations satisfying the $\alpha$-twisted cocycle condition
\[ \varphi_{23} \circ \varphi_{12}= \alpha \cdot \varphi_{13}\]
as morphisms of $L \otimes_k L \otimes_k L$-representations, where $\varphi_{ij} = p_{ij}^* \varphi$. We define the dimension vector of this twisted representation by
\[ \dim(W,\varphi):=\frac{ \dim_L(W)}{\ind \,(\alpha)}, \]
where by the index of $\alpha$, we mean the index of a division algebra $D$ representing the same class in $\Br(k)$.
  
A morphism between two $\alpha$-twisted $k$-representations $(W,\varphi)$ and $(W',\varphi')$ is given by a morphism $\rho : W \lra W'$ of $L$-representations such that $p_2^*\rho \circ \varphi = p_1^* \rho \circ \varphi'$.
\end{defn}

\begin{ex}
If $Q$ is a quiver with one vertex and no arrows, then an $\alpha$-twisted representation of $Q$ over $k$ is an $\alpha$-twisted sheaf over $\spec k$ in the sense of C\u{a}ld\u{a}raru, which we refer to as an $\alpha$-twisted $k$-vector space.
\end{ex}

We define $\crep_k (Q,\alpha)$ to be the category of $\alpha$-twisted $k$-representations of $Q$; one can check that this category does not depend on the choice of \'{e}tale cover on which $\alpha$ trivialises, or on the choice of a representative of the cohomology class of $\alpha$ in $ H^2_{\et}(\spec k, \GG_m)$ analogously to the case for twisted sheaves \textit{cf.}\ \cite[Corollary 1.2.6 and Lemma 1.2.8]{cald}. Furthermore, if the class of $\alpha$ is trivial, then there is an equivalence of categories $\crep_k(Q,\alpha) \cong \crep_k(Q)$. We have the expected functoriality for a field extension $K/k$: there is a functor
\[- \otimes_k K : \crep_k(Q,\alpha) \lra \crep_K(Q,\alpha \otimes_k K)\]
(\textit{cf.}\ \cite[Corollary 1.2.10]{cald} for the analogous statement for twisted sheaves). One can also define families of $\alpha$-twisted representations over a $k$-scheme $T$, where $\alpha$ is the pullback of a $\GG_m$-gerbe on $\spec k$ to $T$ or, more generally, $\alpha$ is any $\GG_m$-gerbe on $T$.

\begin{rmk}\label{Lieblich twisted}
Alternatively, for a $\GG_m$-gerbe $\alpha : \fZ \lra \spec k$ one can define $\alpha$-twisted $k$-representations of $Q$ in an analogous manner to Lieblich \cite{Lieblich} as a tuple $(\cF_v, \varphi_a: \cF_{t(a)} \lra \cF_{h(a)})$ consisting of $\alpha$-twisted locally free coherent sheaves $\cF_v$ and homomorphisms $\varphi_a$ of twisted sheaves, where an $\alpha$-twisted sheaf $\cF$ is a sheaf of $\cO_\fZ$-modules over $\fZ$ whose scalar multiplication homomorphism from this module structure coincides with the action map $\GG_m \times \cF \lra \cF$ coming from the fact that $\fZ$ is a $\GG_m$-gerbe.
\end{rmk}

C\u{a}ld\u{a}raru proves that twisted sheaves can be interpreted as modules over Azumaya algebras. More precisely, for a scheme $X$ and Brauer class $\alpha \in \Br(X)$ that is the class of an Azumaya algebra $\cA$ over $X$, the category of $\alpha$-twisted sheaves $\cmod(X,\alpha)$ over $X$ is equivalent to the category $\cmod(\cA)$ of (right) $\cA$-modules by \cite[Theorem 1.3.7]{cald}. This equivalence is realised by showing that $\cA$ is isomorphic to the endomorphism algebra of an $\alpha$-twisted sheaf $\cE$ (\textit{cf.}\ \cite[Theorem 1.3.5]{cald}) and then 
\[ - \otimes \cE^\vee : \cmod(X,\alpha) \lra \cmod(\cA) \]
gives the desired equivalence. 

Let us describe this equivalence over $X = \spec k$. Let $D$ be a central division algebra over $k$ (or more generally a central simple algebra over $k$). Then $D$ splits over some finite Galois extension $L/k$; that is, there is an isomorphism 
\begin{equation}
j : D \otimes_k L \lra M_n(L),
\end{equation}
where $n = \ind \, (D)$. The isomorphism $j$ and the $\Gal_{L/k}$-action on $L$ and $M_n(L)$ determine a 1-cocycle $a_D :\Gal_{L/k} \lra \PGL_n(L) = \Aut(M_n(L))$ such that $D$ corresponds to $\delta_n(a_D) \in \Br(k)$, where $\delta_n$ is the connecting homomorphism for the short exact sequence $1 \lra \GG_m \lra \GL_n \lra \PGL_n \lra 1$ (for example, see \cite[Theorem 2.4.3]{GS}). More precisely, $D$ is the fixed locus for the twisted $\Gal_{L/k}$-action on $M_n(L)$ defined by the 1-cocycle $a_D$
\begin{equation}
D = (_{a_D} {M_n(L)})^{\Gal_{L/k}}.
\end{equation}
Let $\alpha \in H^2_{\et}(\spec k, \GG_m)$ correspond to $D \in \Br(k)$; then $\alpha$ can be represented by a \v{C}ech cocycle on the \'{e}tale cover given by $L/k$. By \cite[Theorem 1.3.5]{cald}, there is an $\alpha$-twisted $k$-vector space $\cE:=(E,\varphi)$ such that $D$ is isomorphic to the endomorphism algebra of this twisted vector space. Explicitly, we have $E := L^n$ and $\varphi : p_1^*E \lra p_2^*E$ is an isomorphism which induces the isomorphism 
\[ p_1^*\End_L(E) \lra p_1^*(D \otimes_k L) \cong D \otimes_k L \otimes_k L \cong p_2^*(D \otimes_k L) \lra p_2^* \End_L(E), \]
where the first and last maps are pullbacks of the composition of the isomorphism $j : D \otimes_k L \cong M_n(L)$ with the isomorphism $M_n(L) \cong \End(E)$; the existence of such an isomorphism $\varphi$ is given by the Noether-Skolem Theorem and one can check that $\cE:=(E,\varphi)$ is an $\alpha$-twisted sheaf, whose endomorphism algebra is 
\[ \End (\cE) =(  _{a_D}  {\End_L(E)})^{\Gal_{L/k}} \cong ( _{a_D} {M_n(L)})^{\Gal_{L/k}} = D. \]
Then C\u{a}ld\u{a}raru's equivalence is explicitly given by
\begin{equation}\label{cald equiv field}
 - \otimes_k \cE^\vee : \cmod(k,\alpha) \lra \cmod(D).
\end{equation}
With our conventions on dimensions of twisted vector spaces, $\dim \cE = 1$ and the image of this twisted vector space under this equivalence is the trivial module $D$.

For a division algebra $D$, we let $\crep_D(Q)$ denote the category of representations of a quiver $Q$ in the category of $D$-modules.

\begin{prop}\label{equiv twist rep D}
Let $D$ be a central division algebra over a field $k$ and $\alpha$ be a $\GG_m$-gerbe over $k$ representing the same class in $\Br(k)$ as $D$. Then there is an equivalence of categories
\[ F: \crep_k (Q,\alpha) \cong \crep_D(Q).\]
\end{prop}
\begin{proof}
Unravelling Definition \ref{def_twisted_rep}, we see that the category $\crep_k (Q,\alpha)$ is equivalent to the category of representations of $Q$ in the category $\cmod(k,\alpha)$ of $\alpha$-twisted $k$-vector spaces, which we denote by $Q-\cmod(k,\alpha)$. By \cite[Theorem 1.3.7]{cald}, there is an equivalence
\[ \cmod(k,\alpha) \cong \cmod(D) \]
as described in \eqref{cald equiv field} above. Hence, we deduce equivalences
\[ \crep_k (Q,\alpha) \cong Q-\cmod(k,\alpha) \cong Q - \cmod(D)\]
and by definition $\crep_D(Q):= Q - \cmod(D)$.
\end{proof}

There is a natural notion of $\theta$-(semi)stability for twisted representations of $Q$, which involves checking the usual slope condition for twisted subrepresentations, where the dimension of a twisted quiver representation is given in Definition \ref{def_twisted_rep}. By using the functoriality of twisted quiver representations for field extensions, we can also define $\theta$-geometric stability.

\begin{lemma}\label{lemma_stab_twisted_reps_div}
Under the equivalence $F$ of Proposition \ref{equiv twist rep D}, if $\cW:=(W,\varphi)$ is an $\alpha$-twisted representation of $Q$ over $k$, then
\[ \dim_D( F(\cW) )= \dim(\cW). \] 
Moreover, $\theta$-(semi)stability (resp.\ $\theta$-geometric stability) of a twisted representation $\cW$ is equivalent to $\theta$-(semi)stability (resp.\ $\theta$-geometric stability) of the corresponding $D$-representation $F(\cW)$ of $Q$. 
\end{lemma}
\begin{proof} 
The first claim follows by construction of the equivalence of $F$ using the $\alpha$-twisted $k$-vector space $\cE=(L^n,\varphi)$ as in \eqref{cald equiv field}: as we already observed, $D$ is the image of $\cE$ under \eqref{cald equiv field} and $\dim(\cE) = 1 = \dim_D(D)$. Then the claim about $\theta$-(semi)stability follows from this first claim. For the preservation of geometric stability, we note that for any field extension $K/k$ we have a commutative diagram
\[ \xymatrix@1{ \crep(k,\alpha)  \ar[d]^{- \otimes_k K} \ar[r]^{F_k} & \cmod(D) \ar[d]^{- \otimes_k K} \\ \crep(K,\alpha \otimes_k K)  \ar[r]^{F_K} & \cmod(D \otimes_k K)  }\]
and also $F_L$ preserves $\theta$-(semi)stability, by a similar argument.
\end{proof}

We can now reinterpret the rational points of the moduli space $\Modgs$ as twisted quiver representations.

\begin{thm}\label{decomp_twisted_reps}
Let $k$ be a perfect field; then $\Modgs(k)$ is the disjoint union over $[\alpha] \in \mathrm{Im}\,(\cT : \Modgs(k) \lra \Br(k))$ of the set of isomorphism classes of $\alpha$-twisted $\theta$-geometrically stable $d'$-dimensional $k$-representations of $Q$, where $d = \ind \,(\alpha) d'$.
\end{thm}
\begin{proof} 
This follows from Theorem \ref{thm_Galois_div_alg_intro}, Proposition \ref{equiv twist rep D} and Lemma \ref{lemma_stab_twisted_reps_div}.
\end{proof}

\subsection{Moduli of twisted quiver representations}

For a $\GG_m$-gerbe $\alpha $ over a field $k$, we let $^\alpha \fM_{Q,d',k}$ denote the stack of $\alpha$-twisted $d'$-dimensional $k$-representations of $Q$. Following Proposition \ref{equiv twist rep D} and Lemma \ref{lemma_stab_twisted_reps_div} (or strictly speaking a version of this equivalence in families), this stack is isomorphic to the stack $ \fM_{Q,d',D}$ of $d'$-dimensional $D$-representations of $Q$, where $D$ is a central division algebra over $k$ corresponding the cohomology class of $\alpha$. 

\begin{prop}\label{descr_tw_stack}
Let $k$ be a field, $D$ be a central division algebra over $k$ and $\alpha : \fZ \lra \spec k$ be a $\GG_m$-gerbe, whose cohomology class is equal to $D$. Then we have isomorphisms
\[^\alpha \fM_{Q,d',k}\cong\fM_{Q,d',D} \cong [\rep_{Q,d',D}/\mathbf{G}_{Q,d',D}] \]
where the $k$-varieties $\rep_{Q,d',D}$ and $\mathbf{G}_{Q,d',D}$ are constructed in Proposition \ref{exists_k_var_for_D_reps}.
\end{prop}
\begin{proof}
We have already explained the first isomorphism. For the second, we use the fact that there is a tautological family of $d'$-dimensional $D$-representations of $Q$ over $\rep_{Q,d',D}$ which is obtained by Galois descent for the tautological family over $\rep_{Q,d,k^s}$, where $k^s$ is a separable closure of $k$ and $d: = \ind\,(D)d'$.
\end{proof}

As described in $\S$\ref{sec gerbes}, the type map $\cT : \Modgs(k) \lra \Br(k)$ extends to a map $\cG : [\Rep/\Gbar](k) \lra \Br(k)$ defined in \eqref{def_G}. If a division algebra $D$ lies in the image of $\cG$, then it also lies in the image of $\cP: [\Rep/\Gbar](k) \lra H^1_{\et}(\spec k, \Gbar)$. Then we can use the corresponding $\Gbar(k^s)$-valued 1-cocycle on $\Gal_k$ to modify the $\Gal_k$-action on $\rep_{Q,d,k^s}$ in order to obtain the $k$-varieties $\rep_{Q,d',D}$ and $\mathbf{G}_{Q,d',D}$ with $d = \ind \,(D) d'$ analogously to Proposition \ref{exists_k_var_for_D_reps}, where here $k^s$ denotes a separable closure of $k$.

We recall that a $k$-form of a $k^s$-scheme $X$ is a $k$-scheme $Y$ such that $X \cong Y \times_k k^s$. For a central division algebra $D$ over $k$ and dimension vectors $d, d'$ such that $d = \ind\,(D) d'$, the $k$-variety $\rep_{Q,d',D}$ (resp.\ $\mathbf{G}_{Q,d',D}$) is a $k$-form of the affine scheme $\rep \times_k \, k^s = \rep_{Q,d,k^s}$ (resp.\ the reductive group $\mathbf{G} \times_k k^s =  \mathbf{G}_{Q,d,k^s}$), as already seen in the proof of Proposition \ref{exists_k_var_for_D_reps} (see also Remark \ref{rmk nonperf field}). In particular, $\mathbf{G}_{Q,d',D}$ is reductive, as its base change to $k^s$ is reductive.

The following result and Theorem \ref{forms_of_moduli_sp} can be viewed as quiver versions of analogous statements for twisted sheaves due to Lieblich (\textit{cf.}\ \cite[Proposition 3.1.2.2]{Lieblich}). 

\begin{prop}\label{forms_of_stack}
For a field $k$ with separable closure $k^s$, the moduli stack $\fM_{Q,d,k^s}$ has different $k$-forms given by the moduli stacks $^\alpha \fM_{Q,d',k}$ for all $\alpha$ in the image of the map
\[ \cG : [\Rep/\Gbar](k) \lra \Br(k), \]
where $d =  \ind\,(\alpha) d'$.
\end{prop}

\begin{proof}
By Proposition \ref{descr_tw_stack}, we have $^\alpha \fM_{Q,d',k} \times_k k^s \cong [\rep_{Q,d',D}/\mathbf{G}_{Q,d',D}] \times_k  k^s$, which is isomorphic to
\[  [\rep_{Q,d',D} \times_k k^s/\mathbf{G}_{Q,d',D} \times_k k^s] \cong  [ \rep_{Q,d,k^s}/\mathbf{G}_{Q,d,k^s}] \cong \fM_{Q,d,k^s}\]
by Proposition \ref{exists_k_var_for_D_reps} and Remark \ref{rmk nonperf field}.
\end{proof}

For a central division algebra $D$ over $k$ and dimension vectors $d, d'$ such that $d = \ind\,(D) d'$, we note that the reductive $k$-group $\mathbf{G}_{Q,d',D}$ acts on the $k$-variety $\rep_{Q,d',D}$. We can consider the GIT quotient for this action with respect to the character $\chi_\theta : \mathbf{G}_{Q,d',D} \lra \GG_m$ obtained by Galois descent from the character $ \chi_\theta : \mathbf{G}_{Q,d,\kb} \lra \GG_{m,\kb}$. Since $\rep_{Q,d',D} \times_k \kb \cong \rep_{Q,d,\kb}$ and base change by field extensions preserves the GIT (semi)stable sets, we have
\[ \rep_{Q,d',D}^{\chi_\theta-(s)s} \times_k \kb =  \rep_{Q,d,\kb}^{\chi_\theta-(s)s} = \rep_{Q,d,\kb}^{\theta-(s)s}, \]
where the last equality uses the Hilbert--Mumford criterion and we recall that over the algebraically closed field $\kb$ the notions of $\theta$-geometrically stability and $\theta$-stability coincide. By Lemma \ref{lem_geom_s_div_alg} and Lemma \ref{stability_over_separably_closed_field_implies_geometric_stability} and the fact that the GIT (semi)stable sets commute with base change by field extensions, we deduce that
\[ \rep_{Q,d',D}^{\chi_\theta-ss} = \rep_{Q,d',D}^{\theta-ss} \quad \text{and} \quad \rep_{Q,d',D}^{\chi_\theta-s} = \rep_{Q,d',D}^{\theta-gs}. \]
Then we have a GIT quotient 
\[\rep_{Q,d',D}^{\chi_\theta-ss}  \lra \cM^{\theta-ss}_{Q,d',D}:=\rep_{Q,d',D} /\!/_{\chi_\theta} \mathbf{G}_{Q,d',D} \]
that restricts to a geometric quotient $$\rep_{Q,d',D}^{\theta-gs} \lra \cM^{\theta-gs}_{Q,d',D}:=\rep_{Q,d',D}^{\theta-gs}/ \mathbf{G}_{Q,d',D}.$$ We can now prove Theorem \ref{forms_of_moduli_sp}.

\begin{proof}[Proof of Theorem \ref{forms_of_moduli_sp}]
The first statement is shown analogously to the fact that the moduli stack $\fM_{Q,d}^{\theta-gs} = [\Rep^{\theta-gs}/\G]$ of $\theta$-geometrically stable $d$-dimensional $k$-representations of $D$ is a $\GG_m$-gerbe over $$\Modgs = [\Rep^{\theta-gs}/\Gbar].$$

One proves that $ \cM^{\theta-gs}_{Q,d',D}$ corepresents the moduli functor of $\theta$-geometrically stable  $d'$-dimensional  $D$-representations of $Q$ by modifying the argument in Theorem \ref{GIT_const_of_ModQd}. More precisely, we obtain a tautological family on $\rep_{Q,D,d'}$ from Galois descent, using the tautological family on $\rep_{Q,d,k^s}$. Then by Proposition \ref{equiv twist rep D} and Lemma \ref{lemma_stab_twisted_reps_div}, we see that  $\cM^{\theta-gs}_{Q,d',D}$ also corepresents the second moduli functor.

The final statement follows as in Proposition \ref{forms_of_stack}.
\end{proof}

In Theorem \ref{forms_of_moduli_sp}, we emphasise that the term coarse moduli space is used in the sense of stacks. In particular, we note that the $k$-rational points of  $\cM^{\theta-gs}_{Q,d',D}$ are not in bijection with the set of isomorphism classes of $d'$-dimensional $\theta$-geometrically stable $D$-representations of $Q$ in general (as we have already observed for the trivial division algebra $D = k$). 

There are some natural parallels between the results in this section and the work of Le Bruyn \cite{LeBruyn_Dbranes}, who describes the $A$-valued points of the moduli stack $\fX = [\Rep/\Gbar]$ over $ k = \spec \CC$, for a commutative $\CC$-algebra $A$, in terms of algebra morphisms from the quiver algebra $\CC Q$ to an Azumaya algebra $\cA$ over $A$. He also relates these $A$-valued points to twisted quiver representations. In this case, for $A = \CC$, there are no twisted representations as $\Br(\CC) = 1$, whereas for non-algebraically closed field $k$, we see twisted representations as $k$-rational points of $\fX$. By combining the ideas of \cite{LeBruyn_Dbranes} with the techniques for non-algebraically closed fields $k$ used in the present paper, it should be possible to also describe the $A$-valued points of the moduli stack $\fX = [\Rep/\Gbar]$ over an arbitrary field $k$ for any commutative $k$-algebra $A$.

\subsection{Universal twisted families}

Let us now use twisted representations to describe the failure of $\Modgs$ to admit a universal family of quiver representations and to give a universal twisted representation over this moduli space. In this section, $k$ is an arbitrary field.

\begin{defn}
Let $\alpha(\Modgs) \in H^2_{\et}(\Modgs, \GG_m)$ be the class of the $\GG_m$-gerbe $\pi^{\theta-gs} : \fM^{\theta-gs}_{Q,d} \lra \Modgs$; we refer to $\alpha(\Modgs)$ as the obstruction to the existence of a universal family on $\Modgs$.
\end{defn}

\begin{rmk}\label{rmk obstr class}
Let $\beta(\Modgs) \in H^1_{\et}(\Modgs, \Gbar)$ be the class of the $\Gbar$-torsor $p^{\theta-gs} : \rep^{\theta-gs} \lra \Modgs$. By Lemma \ref{delta universally}, we have
\[ \alpha(\Modgs) = \delta(\beta(\Modgs)) \]
for the connecting homomorphism $\delta: H^1_{\et}(\Modgs, \Gbar) \lra H^2_{\et}(\Modgs, \GG_m)$. 
\end{rmk}

\begin{lemma}
The class $\alpha:=\alpha(\Modgs)$ is a Brauer class.
\end{lemma}
\begin{proof}
We will show that $\alpha \in H^2_{\et}(\Modgs, \GG_m)$ is a Brauer class, by proving that it is the image of an Azumaya algebra of index $N:= \sum_{v \in V} d_v$ on $\Modgs$. The representation $\G \lra \GL_N$, given by including each copy of $\GL_{d_v}$ diagonally into $\GL_N$, descends to a homomorphism $\Gbar \lra \PGL_N$. This gives the following factorisation of $\delta$
\[ \delta : H^1_{\et}(\Modgs, \Gbar) \lra H^1_{\et}(\Modgs, \PGL_N) \lra H^2_{\et}(\Modgs, \GG_m), \]
which proves this claim by Remark \ref{rmk obstr class}.
\end{proof}

The following result explains the name of the class $\alpha(\Modgs)$ given above. This result is a quiver analogue of the corresponding statement for twisted sheaves due to C\u{a}ld\u{a}raru \cite[Proposition 3.3.2]{cald}.

\begin{prop}\label{twisted_univ_family}
Let $\alpha:=\alpha(\Modgs)$ denote the obstruction class to the existence of a universal family. Then there is a \lq universal' $\alpha$-twisted family $\cW$ of $\theta$-geometrically stable $k$-representations of $Q$ over $\Modgs$; that is, there is an \'{e}tale cover $\{ f_i: U_i \lra \Modgs \}$ such that over $U_i$ there are local universal families $\cW_i$ of $k$-representations of $Q$ and there are isomorphisms $\varphi_{ij} : \cW_i|_{U_{ij}} \lra \cW_j|_{U_{ij}}$ which satisfy the $\alpha$-twisted cocycle condition:
\[  \varphi_{jl} \circ \varphi_{ij}= \alpha_{ijl} \cdot \varphi_{il}. \]
In particular, $\Modgs$ admits a universal family of quiver representations if and only if the obstruction class $\alpha  \in \Br(\Modgs)$ is trivial.
\end{prop}
\begin{proof}
Let us take an \'{e}tale cover $\{ f_i : U_i \lra \Modgs \}$ on which the principal $\Gbar$-bundle $\cP:= \Rep^{\theta-gs} \lra \Modgs$ is trivialisable: if $\cP_i:=f_i^*\cP$, then we have isomorphisms $\psi_i : \cP_i \cong U_i \times \Gbar$. Let $\cF_i \lra \cP_i$ denote the pullback of the tautological family $\cF \lra \cP:=\Rep^{\theta-gs}$ of $\theta$-geometrically stable $d$-dimensional $k$-representations to $\cP_i$. The family $\cF_i = (\cF_{i,v}, \phi_{i,a}: \cF_{i,t(a)} \lra \cF_{i,h(a)})$ consists of rank $d_v$ trivial bundles $\cF_{i,v}$ over $\cP_i$ with a $\G$-linearisation, such that $\Delta \cong \GG_m$ acts on the fibres with weight $1$. We can modify this family by observing that there is a line bundle $\cL_i \lra \cP_i$ given by the $\Delta$-bundle $U_i \times \G \lra U_i \times \Gbar \cong \cP_i$, which has a $\G$-linearisation, where again $\Delta$ acts by weight $1$. Then the $\Delta$-weight on $\cF'_i:=\cF_i \otimes \cL_i^\vee$ is zero, and thus the sheaves $\cF'_{i,v}$ admit $\Gbar$-linearisations. 

We can now use descent theory for sheaves over the morphism $ \cP_i \lra U_i$ to prove that the family $\cF_i'$ of representations of $Q$ over $\cP_i$ descends to a family $\cW_i$ over $U_i$. More precisely, by  \cite[Theorem 4.2.14]{Huybrechts_Lehn},  the $\Gbar$-linearisation on $\cF'_{i,v}$ gives an isomorphism $\text{pr}_1^*\cF'_{i,v} \cong \text{pr}_2^*\cF'_{i,v}$, for the projections $\text{pr}_i : \cP_i \times_{U_i} \cP_i \lra \cP_i$, and this satisfies the cocycle condition; hence, $\cF'_{i,v}$ descends to a sheaf $\cW_{i,v}$ over $U_i$. The homomorphisms $\phi_{i,a}$ descend to $U_i$ similarly. Since the families $\cW_i$ over $U_i$ descend from the tautological family locally, they are local universal families (for example, locally adapt the corresponding argument for moduli of sheaves in \cite[Proposition 4.6.2]{Huybrechts_Lehn}). 

We can refine our \'{e}tale cover, so that $\Pic(U_{ij}) = 0$. Then, as the local universal families $\cW_{i}$ and $\cW_j$ are equivalent over $U_{ij}$, there is an isomorphism
\[ \varphi_{ij} : \cW_i|_{U_{ij}} \lra \cW_j|_{U_{ij}}. \]
On the triple intersections $U_{ijk}$, we let $\gamma_{ijk}:=  \varphi_{ik}^{-1} \circ  \varphi_{jk} \circ  \varphi_{ij} \in \Aut(\cW_i|_{U_{ijk}})$. As $\cW_i$ are families of $\theta$-geometrically stable representations, which in particular are simple, it follows that $\gamma_{ijk} \in \Gamma(U_{ijk},\GG_m)$. Hence $\cW:=(\cW_i, \varphi_{ij})$ is a $\gamma$-twisted family of $\theta$-geometrically stable $k$-representations of $Q$ over $\Modgs$. 

It remains to check that the classes $\gamma$ and $\alpha$ in $H_{\et}^2(\Modgs,\GG_m)$ coincide. We note that $\alpha := \delta(\beta)$ for the class $\beta:=[\cP] \in H^1_{\et}(\Modgs, \Gbar)$. We can describe the cocycle representing $\beta$ by using our given \'{e}tale cover $\{ f_i : U_i \lra \Modgs \}$ on which $\cP$ is trivialisable. More precisely, $\beta$ is represented by the cocycle given by transition functions $\beta_{ij} \in \Gamma(U_{ij}, \Gbar)$ for $\cP$ such that $s_i = \beta_{ij} s_j$, where $s_i : U_i \lra \cP_i$ are the sections giving the isomorphism $\psi_i$. If we take lifts $\tilde{\beta}_{ij} \in \Gamma(U_{ij}, \G)$ of $\beta_{ij}$, then these determine $\alpha_{ijk} \in \Gamma(U_{ijk}, \GG_m)$ by the relation $\alpha_{ijk} \tilde{\beta}_{ik} =\tilde{\beta}_{jk} \tilde{\beta}_{ij}$ over $U_{ijk}$. By pulling back the isomorphisms $\varphi_{ij}$ along the $\G$-invariant morphisms $\cP_{ij} \lra U_{ij}$, we obtain isomorphisms $\cL_i|_{U_{ij}} \cong \cL_j|_{U_{ij}}$ as $\G$-bundles over $U_{ij}$, which is given by a section $\eta_{ij} \in \Gamma(U_{ij}, \G)$. By construction $\gamma = \delta([\bar{\eta}_{ij}])$, where $\bar{\eta}_{ij} \in \Gamma(U_{ij}, \Gbar)$ is the image of $\eta_{ij}$. Since $\eta_{ij}$ are also lifts of the cocycle $\beta_{ij}$, it follows that $\gamma = \alpha$.
\end{proof}

We can consider $\cW$ as a universal $\alpha$-twisted family over $\Modgs$ of $k$-representa\-tions of $Q$. In particular, if the obstruction class $\alpha(\Modgs)$ is trivial, then $\Modgs$ is a fine moduli space, as it admits a universal family. We note that if $r : \spec k \lra \Modgs$, then the image of $\alpha(\Modgs)$ under the map
\[ r^* : H^2_{\et}(\Modgs,\GG_m) \lra H^2(\spec k, \GG_m) \]
is the class $\cG(r)$ described in $\S$\ref{sec gerbes} and the index of the central division algebra corresponding to $\cG(r) \in \Br(k)$ divides the dimension vector $d$ by Proposition \ref{nec_con_div_alg}. 

If the dimension vector $d$ is primitive, then the moduli space $\Modgs$ is fine by  \cite[Proposition 5.3]{king}. The Brauer group of moduli spaces of quiver representations was studied by Reineke and Schroer \cite{Reineke_Schroer}; for several quiver moduli spaces, they describe the Brauer group and prove the non-existence of a universal family in the case of non-primitive dimension vectors \textit{cf}.\ \cite[Theorem 3.4]{Reineke_Schroer}. Proposition \ref{twisted_univ_family} offers some compensation for this seemingly negative result: instead, one has a twisted universal family.

\section{Fields of moduli and fields of definition}\label{sec_field_of_moduli_and_domain_of_def}

The goal of this section is to prove Theorem \ref{field_of_moduli_and_domain_of_def}. First, let us recall a few facts about the notions of field of definition and field of moduli, in the context of quiver representations. Let $\Omega$ be a field and let $\ov{\Omega}$ be an algebraic closure of $\Omega$. One says that a $\Omega$-representation $W$ is \textit{defined} over a subfield $L\subset \Omega$ (or that $L$ is a field of definition of $W$), if there exists an $L$-representation $W'$ such that $\Omega\otimes_{L} W' \simeq W$ as $\Omega$-representations. And one says that $W$ is \textit{definable} over $L$ if there exists an $L$-representation $W'$ such that $\ov{\Omega}\otimes_{L} W' \simeq \ov{\Omega} \otimes_{\Omega} W$ as $\ov{\Omega}$-representations. The \textit{field of moduli} of an $\Omega$-representation $W$ is the intersection of all fields $L\subset\Omega$ over which $W$ is definable. In particular, the field of moduli of the $\Omega$-representation $W$ is equal to the field of moduli of the $\ov{\Omega}$-representation $\ov{\Omega}\otimes_{\Omega} W$.

For the purpose of proving Theorem \ref{field_of_moduli_and_domain_of_def}, it is sufficient to work in a relative and slightly simplified setting: we fix a base field $k$, an algebraic closure $\kb$, and consider only the case $\Omega=\kb$. Given a $\kb$-representation $W$, we let $\mathrm{Def}_k(W)$ be the set of all intermediate fields $k\subset L \subset \kb$ over which $W$ is definable. The field of moduli of $W$ is then equal to $$k_W := \bigcap_{L\in \mathrm{Def}_k(W)} L.$$ By definition, $k\subset k_W \subset \kb$; thus if $k$ is perfect, so is $k_W$. Consider the group $$H_k(W) := \{\tau\in \Aut(\kb/k) \ |\ W^{\tau} \simeq W\}$$ where $W^\tau$ is the representation of $Q$ defined from $W$ and $\tau$ as in Section \ref{rational_pts_section}. The following result explains why the field of moduli is sometimes defined, in other contexts as well, as $\kb^{H_k(W)}$.

\begin{thm}\label{field_of_moduli_as_fixed_pt_field}
Let $W$ be a $\theta$-stable $\kb$-representation of $Q$. The field $\kb^{H_k(W)}$ is a purely inseparable extension of $k_W$. In particular, if $k$ is perfect, then $$k_W= \kb^{H_k(W)}.$$
\end{thm}

To prove Theorem \ref{field_of_moduli_as_fixed_pt_field}, it suffices to show that $\tau\in\Aut(\kb/k)$ satisfies $W^\tau\simeq W$ if and only if $\tau|_{k_W} = \mathrm{Id}_{k_W}$. First, we give a useful characterisation of $k_W$.

\begin{lemma}\label{field_of_def_of_a_stable_orbit}
Let $W$ be a $\theta$-stable $d$-dimensional $\kb$-representation; then $k_W$ is the minimal field of definition $L_W$ of the corresponding $\G(\kb)$-orbit $O_W\subset \Reps(\kb)$.
\end{lemma}

\begin{proof}
Note that $O_W$ is a closed subvariety of $\Reps(\kb)$, as $W$ is GIT-stable; thus $O_W$ has a minimal field of definition $L_W$ by \cite[Proposition 3.11]{Vojta}. Let $O'_W$ be an $L_W$-variety with $O'_W \times_{L_W} \kb \cong O_W$; then $L_W$ is contained in the intersection of all residue fields of closed points of $O'_W$ (as via the projection $O_W \ra O_W'$ the residue field of any closed point in $O'_W$ is a finite extension of $L_W$ and a subfield of $\kb$). In fact, $L_W$ is the intersection of all these residue fields, as $O_W$ is a $\G(\kb)$-orbit and $\G(\kb)$ is defined over $k$. A closed point $x' \in O'_W$ lifts to $x \in O_W$, corresponding to a $\kb$-representation isomorphic to $W$, thus $W$ is definable over $\kappa(x')$ and
$$k_W = \bigcap_{L\in\mathrm{Def}_k(W)} L \quad  \subset  \quad \bigcap_{p \in O'_W} \kappa(p) = L_W.$$ Conversely, if $L\in\mathrm{Def}_k(W)$, then there exists an $L$-representation $W'$ such that $W\simeq \kb\otimes_L W'$, so there is $p\in O'_W$ such that $\kappa(p)\simeq L$. Hence $L_W \subset L$. As this is true for all $L\in\mathrm{Def}_k(W)$, we have $L_W\subset k_W$. 
\end{proof}

\begin{lemma}\label{Galois_gp_of_kb_over_L_W}
With the same assumptions as in Lemma \ref{field_of_def_of_a_stable_orbit}, let $\tau\in\Aut(\kb/k)$. Then $W^\tau\simeq W$ if and only if $\tau|_{L_W} = \mathrm{Id}_{L_W}$, i.e.\ $H_k(W) \simeq \Aut(\kb/L_W)$.
\end{lemma}

\begin{proof}
If $W^\tau\simeq W$ (i.e.\ $\tau\in H_k(W)\subset \Aut(\kb/k)$), then $O_W = O_{W^\tau} = \tau(O_W)$ in $\Reps(\kb)$. By descent theory with respect to the Galois extension $\kb/\kb^{H_k(W)}$, we then have that $O_W$ is defined over $\kb^{H_k(W)}$. So $L_W\subset \kb^{H_k(W)}$. In particular, for all $\tau\in H_k(W)$, one has $\tau|_{L_W} = \mathrm{Id}_{L_W}$. Conversely, let us assume that $\tau\in\Aut(\kb/k)$ satisfies $\tau|_{L_W} = \mathrm{Id}_{L_W}$. Since $O_W$ is defined over $L_W$, this implies that $\tau(O_W) = O_W$, which in turn implies that $W^\tau\simeq W$.
\end{proof}

We can now prove Theorems \ref{field_of_moduli_as_fixed_pt_field} and Theorem \ref{field_of_moduli_and_domain_of_def}.

\begin{proof}[Proof of Theorem \ref{field_of_moduli_as_fixed_pt_field}]
By Lemma \ref{field_of_def_of_a_stable_orbit}, $k_W= L_W$. And by Lemma \ref{Galois_gp_of_kb_over_L_W}, $H_k(W) \simeq \Aut(\kb/L_W)$. So $\kb^{H_k(W)}\simeq \kb^{\Aut(\kb/L_W)}$ is a purely inseparable extension of $L_W$.
\end{proof}

\begin{proof}[Proof of Theorem \ref{field_of_moduli_and_domain_of_def}]
Recall that here $k$ is assumed to be perfect and for the moduli functor $F_{Q,d}^{\theta-gs}$ from \eqref{moduli_functors}, we have, by Theorem \ref{GIT_const_of_ModQd}, a bijective map $$\pi_k: F_{Q,d}^{\theta-gs}(\kb) \overset{\simeq}{\lra} \Modgs(\kb) = \Rep^{\theta-gs}(\kb)/\G(\kb).$$ If $W$ is a $\theta$-stable $d$-dimensional $\kb$-representation of $W$, we denote by $O_W$ the associated $\kb$-point of the $k$-variety $\Modgs$, via the map $\pi_k$. We have seen in Section \ref{rational_pts_section} that $\Gal_k$ acts on $\Modgs(\kb)$ and that $\pi_k$ is $\Gal_k$-equivariant. Therefore, as $O_W$ represents the isomorphism class of the $\kb$-representation $W$, we have that $$H_k(W) \simeq 
\{\tau\in\Gal_k\ |\ \tau(O_W)=O_W\}=:\mathrm{Stab}_{\Gal_k}(O_W),$$ which implies that $\kb^{H_k(W)} \simeq \kappa(O_W)$, the residue field of the point $O_W$ of $\Modgs$. Theorem \ref{field_of_moduli_as_fixed_pt_field} then shows that $\kappa(O_W)\simeq k_W$. The rest of Theorem \ref{field_of_moduli_and_domain_of_def} then follows immediately by base changing to the field $k_W$ and applying Theorem \ref{fibres div alg} to this perfect field: the central division algebra $D_W$ over $k_W$ is the image of $O_W$ under the type map $\mathcal{T}$ and $O_W\in \mathcal{T}^{-1}(D_W)$ is an isomorphism class of a $\theta$-geometrically stable $D_W$-representation of $Q$.
\end{proof}

As the type map can be described via the $\GG_m$-gerbe $\fM_{Q,d}^{\theta-gs} \lra \Modgs$ by Corollary \ref{two_type_maps_agree}, we immediately deduce the following cohomological obstruction to representations being defined over their field of moduli.

\begin{cor}\label{cor_coh_obs_define_field_of_moduli}
Keeping the assumptions and notation of Theorem \ref{field_of_moduli_and_domain_of_def}, the $\kb$-representation $W$ is definable over its field of moduli $k_W = \kappa(O_W)$ if the class of the $\GG_m$-gerbe $p_W^* \fM_{Q,d}^{\theta-gs}$ (given by the pullback of the gerbe $\fM_{Q,d}^{\theta-gs} \lra \Modgs$ along the point $p_W : \spec \kappa(O_W) \ra \Modgs$) is trivial in $H^2_{\et}(\spec k_W, \GG_m)$.
\end{cor}

We saw in Example \ref{quaternionic_rep} how to exhibit $\theta$-stable $\kb$-representations of $Q$ that are not definable over their field of moduli but only over a central division algebra over that field.


\def\cprime{$'$}

\end{document}